\documentclass[a4paper,10pt]{article}

\usepackage{a4wide}
\usepackage{graphicx,xcolor,textpos}
\usepackage{helvet}
\usepackage{amsfonts}
\usepackage{amssymb}
\usepackage{amsmath}
\usepackage{amsthm}
\usepackage{graphicx}
\usepackage{subcaption}
\usepackage{multicol}
\usepackage{siunitx}
\usepackage{url}
\usepackage[utf8]{inputenc}
\usepackage[nottoc]{tocbibind}
\usepackage{tikz}
\usepackage[symbol]{footmisc}
\usepackage{xcolor}
\usepackage{mleftright}
\usepackage{verbatim}
\usepackage[ruled,vlined]{algorithm2e}
\usepackage[noend]{algpseudocode}
\usepackage{lipsum}
\usepackage{mathrsfs}
\usepackage{tikz-cd}
\usepackage{hyperref}
\hypersetup{
	colorlinks=true,
	linkcolor=blue,
	filecolor=magenta,      
	urlcolor=cyan,
}
\usepackage{tabularx}
\usepackage{multirow}
\usepackage{longtable}
\usepackage{emptypage}
\usepackage{dsfont}
\usepackage{stackengine}

\newtheorem{theorem}{Theorem}

\newtheorem*{theorem*}{Theorem}
\newtheorem{duptheorem}{Theorem}

\newtheorem{lemma}{Lemma}[section]
\newtheorem{loanthm}[lemma]{Theorem}
\newtheorem*{definition*}{Definition}
\newtheorem{proposition}[lemma]{Proposition}
\newtheorem*{proposition*}{Proposition}
\newtheorem{example}[lemma]{Example}
\newtheorem{corollary}[lemma]{Corollary}
\newtheorem{remark}[lemma]{Remark}

\newtheorem*{convention}{Convention}

\newtheorem{definition}[lemma]{Definition}
\newtheorem{notation}[lemma]{Notation}

\DeclareMathOperator{\Gal}{Gal}

\newcommand{\QQ}{\mathbb{Q}}
\newcommand{\RR}{\mathbb{R}}
\newcommand{\CC}{\mathbb{C}}
\newcommand{\ZZ}{\mathbb{Z}}

\newcommand{\NN}{\mathbb{N}}

\newcommand{\id}{\mathds{1}}

\newcommand{\inv}{^{-1}}
\newcommand{\comp}{\circ}

\newcommand{\abs}[1]{\left|#1\right|}
\newcommand{\norm}[1]{\left|\left|#1\right|\right|}
\newcommand{\set}[1]{\left\{#1\right\}}

\newcommand{\acts}{\curvearrowright}

\newcommand{\mf}[1]{\mathfrak{#1}}
\newcommand{\mc}[1]{\mathcal{#1}}

\newcommand{\kgv}{\mathrm{lcm}}

\newcommand{\without}[1]{\backslash\{#1\}}

\newcommand{\normal}{\vartriangleleft}
\newcommand{\conj}{\mathrm{Conj}}

\newcommand{\betha}{\beta}

\newcommand{\trace}{\mathrm{Tr}}

\newcommand{\range}[1]{\{1,2,\cdots,#1\}}

\newcommand{\ZZG}{\ZZ[\zeta_G]}
\newcommand{\ZZp}{\frac{\ZZ}{p\ZZ}}

\DeclareMathOperator{\Ima}{Im}
\DeclareMathOperator{\RF}{RF}

\title{Effective conjugacy separability of virtually abelian groups.}
\author{Jonas Der\'e and Lukas Vandeputte \thanks{Corresponding author: Jonas Der\'e (jonas.dere@kuleuven.be). The first author were supported by Internal Funds KU Leuven (project number 3E220559). }}
\date{KU Leuven Campus Kulak Kortrijk\\ Department of Mathematics\\
	Etienne Sabbelaan 53\\
	B-8560 Kortrijk, Belgium. }
\begin{document}
	\sloppy
	\maketitle
	\begin{abstract}
	A natural question for groups $H$ is which data can be detected in its finite quotients. A subset $X \subset H$ is called separable if for all $h\in H \setminus X$, there exists an epimorphism $\varphi$ to a finite group $Q$ such that $\varphi(h)\notin\varphi(X)$. More specifically, a group is said to be conjugacy separable if every conjugacy class is separable. It is known that many classes of groups are conjugacy separable, including virtually free and polycyclic groups. The minimal order of the quotient $Q$, in terms of the complexity of the conjugacy classes under consideration, is captured by the conjugacy separability function $\conj_H: \mathbb{N} \to \mathbb{N}$. This function is in general ill understood, in fact the only large class of groups for which it is known exactly are the abelian groups. Indeed, in this case $\conj_H$ is equal to the residual finiteness function, that is the size of quotients needed to separate singletons, and thus logarithmic if the group is infinite.
	
	Recent work has described the residual finiteness function for the class of virtually abelian groups, which gives a lower bound for the conjugacy separability function. The main result of this paper is a characterization of $\conj_H$ for every virtually abelian group $H$. If the corresponding extension is associated with an irreducible representation over $\QQ$, we demonstrate that we obtain the same function as the residual finiteness function. However, if the representation is not irreducible, we find an expression that is in some cases strictly larger, which we illustrate with several examples.
	\end{abstract}
	\section{Introduction}
	We say a group $H$ is residually finite if for all non-trivial $h\in H$, there exists an epimorphism to a finite group $\varphi: H\rightarrow Q$ preserving non-triviality of $h$. It was shown by Mal'cev in \cite{mal1965faithful} that if a finitely presented group is residually finite, then it has solvable word problem. Similarly, a group $H$ is conjugacy separable if for non-conjugate $g,h\in G$, there exists an epimorphism to a finite group $\varphi:H\rightarrow Q$ such that $\varphi(g)$ and $\varphi(h)$ are non-conjugate. Again if a finitely presented group is conjugacy separable, it has solvable conjugacy problem. Classically, the main focus is on determining which groups satisfy these separability properties. For example it has been shown that free groups, virtually polycyclic groups, surface groups, and fundamental groups of closed orientable $3$-manifolds are all residually finite and conjugacy separable, see for example the list of papers \cite{blackburn1965conjugacy,formanek1976conjugate,hamilton2013separability,hempel1987residual,remeslennikov1971finite,stebe1972conjugacy}
	
However, in recent years, an effort was made in not only determining which groups are residually finite or conjugacy separable, but also quantifying to which degree this is the case. Bou-Rabee introduced the residual finiteness function ${\mathrm{RF}_H:\NN\rightarrow\NN}$ in \cite{bou2010quantifying}, which returns for every $n\in\NN$ the minimal integer ${\mathrm{RF}}_H(n)$ such that for every $h\in H$ with $\norm{h}_S\leq n$, there exists some quotient $Q$ as above with $\abs{Q}\leq{\mathrm{RF}}_H(n)$. This function depends on the word norm $\norm{.}_S$ and thus on the choice of generating set $S$, but the asymptotic growth of this function is well-defined. In the same spirit, Lawton, Louder and McReynolds defined the function $\conj_H: \mathbb{N} \to \mathbb{N}$ in \cite{LLM} which again for any $n\in\NN$, returns the minimal integer $\conj_H(n)$ such that for all non-conjugate $a,b\in H$ with $\norm{a}_S,\norm{b}_S\leq n$, there exists a finite quotient $\varphi:H\rightarrow Q$ such that $\varphi(a)$ and $\varphi(b)$ are non-conjugate and $\abs{Q}\leq \conj_H(n)$.
	
	These functions play a crucial role in other areas such as geometry, topology, arithmetic and dynamics, see \cite{survey} for a recent survey on this topic. For the residual finiteness growth, there are several known upper bounds, where one the most general states that for any finitely generated linear group $H$ the function $\RF_H(n)$ is at most polynomial, see \cite{MR3323650}, with a conjectural value of $n^{\frac{3}{2}}$ for free groups. For some specific classes groups, this bound is strengthened, for instance when $H$ is virtually nilpotent then $\RF_H(n)\prec \ln(n)^d$ by \cite{bou2010quantifying}. In contrast to upper bounds, lower bounds are much rarer. By consequence, there are not a lot of groups for which the exact function $\RF_H$ is known, but examples include among others the finitely generated abelian groups, where we obtain $\ln(n)$ by \cite{bou2010quantifying}, certain nilpotent groups \cite{pengitore2015effective}, some subgroups of Chevalley groups in \cite{bou2012quantifying,MR3633299}. Very recently $\RF_H$ was determined for virtually abelian groups $H$ in \cite{dere2023residual}.
	
	For conjugacy separability function, there are even fewer results, and almost no exact bounds. In some specific cases there are known upper and lower bounds to these functions, but only in rare cases do these bounds match. So far this is only the case if either the group is finite, in which case the conjugacy separability function is constant, if the group is infinite and abelian, giving a logarithmic conjugacy separability function by \cite{bou2010quantifying}, when the group is a lamplighter group in which case we obtain an exponential function by \cite{ferov2022quantifying} or when the group is a generalised Heisenberg group, in which case we obtain a polynomial of a certain degree, see \cite{pengitore2015effective}. Despite their importance, these separability functions are still quite ill understood.
	
As mentioned above, recent work of the first author together with Matthys managed to fully determine the residual finiteness function for the class of finitely generated virtually abelian groups.
	\begin{loanthm}[\cite{dere2023residual}]
		Let $H$ be a finitely generated virtually abelian group of Hirsch length $h$, then there exists a $0 \leq k \leq h$ such that $$
		\mathrm{RF}_H(n)\simeq \ln^k(n).
		$$
	\end{loanthm}
\noindent	The constant $k$ in this theorem can be computed explicitly, and is closely related to the $\CC$-irreducible subrepresentations of an induced representation $\rho: G \acts \ZZ^h$ of some finite group $G$, with $h$ the Hirsch length of $H$.

	In this paper we study the conjugacy separability function for the same class of groups. Our main result looks very similar to the one above, but then for the conjugacy separability function.	
	\setcounter{duptheorem}{0}
	\begin{duptheorem}
		Let $H$ be a finitely generated virtually abelian group of Hirsch lenght $h$, then there exists a $0 \leq k \leq h$ such that $$
		\conj_H(n)\simeq \ln^k(n).
		$$
	\end{duptheorem}
\noindent This is a consequence of Theorem \ref{prop:conjugacyseparabilittysplitvirtuallyabelian:theorem} which gives a more precise characterisation of the constant $k$ in terms of the $\QQ$-irreducible subrepresentations of $\rho: G \acts \ZZ^h$ as mentioned above. However, the exact statement of this theorem is quite technical and as such will only be presented in full detail in Section \ref{sec:result}.
	
		In case the representation $\rho$ is irreducible over $\QQ$, we obtain that the conjugacy separability function is equivalent to the residual finiteness function as obtained in $\cite{dere2023residual}$. The residual finiteness function always forms a lower bound to the conjugacy separability function, but if the representation $\rho$ is not irreducible, then we give concrete examples showing that the conjugacy separability function may be strictly larger at the end of the same section.
	
	Section \ref{sec:prereq} will introduce some necessary background information, both on effective separability and on representation theory. Next, Section \ref{sec:result} is dedicated to the precise formulation of Theorem \ref{prop:conjugacyseparabilittysplitvirtuallyabelian:theorem}, followed by several consequences and examples of this result. The remaining sections are dedicated to proving Theorem \ref{prop:conjugacyseparabilittysplitvirtuallyabelian:theorem}, first by comparing the two different constants in the formulation and then by showing the lower and upper bound, respectively.
	\section{Preliminaries}\label{sec:prereq}
In this section, we recall some elementary notions about on the one hand separability and on the other hand representations of finite groups. Most details can be found in \cite{survey,isaacs2006character,MR3729310}.
	\numberwithin{lemma}{subsection}
	\subsection{Separability}
	First we introduce the necessary background about conjugacy separability. 
	\begin{definition}
		Let $H$ be a group with elements $h_1,h_2\in H$. We call $h_1$ and $h_2$ \textbf{conjugate}, denoted with $h_1\sim h_2$, if and only if there exists some $g\in H$ such that $gh_1g\inv=h_2$. The group $H$ is called \textbf{conjugacy separable}, if for every pair of elements $h_1,h_2\in H$ such that $h_1\not\sim h_2$, there exists some normal subgroup $N\normal H$  of finite index in $H$, such that $h_1N\not\sim h_2N$ in $H / N$.
	\end{definition}
	To associate a growth function to conjugacy separability on a finitely generated group $H$, we use the word norm.
	
	\begin{definition}
		Let $S$ be a finite generating set of the group $H$, then we define the \textbf{word norm} 
		$$\norm{h}_S=\begin{cases}\begin{matrix}
				0 &\text{if } g=e_H\\
				\min\set{n\in\NN\mid h=s_1s_2\cdots s_n\text{ with } s_i\in S\cup S\inv} &\text{otherwise.}
			\end{matrix}
		\end{cases}$$
	\end{definition}
	The word norm depends on the choice of generating set, but only by at most a constant.

	\begin{lemma}
		\label{prop:normAreEquiv}
		Let $H$ be a finitely generated group and let $S$ and $T$ be two finite generating subsets of $H$. There exists a constant $C > 0$ such that for every element $h \in H$ the following holds: $$\norm{h}_T\leq C\norm{h}_S $$
	\end{lemma}
	A proof for this can be found in for instance \cite{MR3729310}, as well as more general information on the word norm. The word metric allows us to talk about the following two functions.
	\begin{definition}
		Let $H$ be a conjugacy separable group, then the function $\conj_H: H\times H\rightarrow \NN:$ is defined as 
		$$
		\conj_{H}(h_1,h_2)=\begin{cases}
			\begin{matrix}
				1&\text{if } h_1\sim h_2\\
				\min\left\{[H:N]\mid N\normal H:h_1N\not\sim h_2N\right\}&\text{otherwise.}
			\end{matrix}
		\end{cases}
		$$
		If $H$ is moreover finitely generated by $S$, then we define the \textbf{conjugacy separability growth function} $\conj_{H,S}: \NN_0\to\NN_0$ as 		$$
		\conj_{H,S}(n) =  \max\set{\conj_H(h_1,h_2)\mid h_1,h_2\in H,\norm{h_1}_S\leq n,\norm{h_2}_S\leq n}.
		$$
	\end{definition}
	This function was introduced first in \cite{LLM}. As the closed balls are finite, conjugacy separability guarantees that $\conj_{H,S}$ only takes finite values and is thus well defined. We will study $\conj_{H,S}$ up to the following equivalence relation.
	
	\begin{definition}
		Let $f,g:\NN_0\rightarrow\RR^+_0$ be non-descending functions. Then we write $f\prec g$ if there exists some natural constant $C$ such that $f(n)\leq Cg(Cn)$ for all $n\in \NN_0$. Furthermore if both $f\prec g$ and $g\prec f$, then we write $f\simeq g$.
	\end{definition}
	\begin{remark}
	Throughout the paper, we will often compare functions that also take negative values. If $f,g:\NN_0\rightarrow\RR$ are non-decreasing functions, then we denote $f \prec g$ if $\max(1,f)\prec\max(1,g)$, and similarly for $f \simeq g$. 
	\end{remark}
From the definitions above, the following is immediate (see \cite[Lemma 1.1]{bou2010quantifying}).
	\begin{lemma}\label{prop:generatorinvariance}
		Let $H$ be a finitely generated conjugacy separable group with finite generating sets $S$ and $T$, then $\conj_{H,S}\simeq\conj_{H,T}$.
	\end{lemma}
\noindent Hence, the asymptotic behavior of the conjugacy separability function is an invariant of the group, and we will often write $\conj_{H,S}$ as $\conj_H$ if we study the asymptotics.

It is often convenient to only compute these functions at a certain values. The following lemmas will be helpful to compare functions in that perspective.
		\begin{lemma}\label{prop:logarithmicLowerBound:lemma}
		Let $f:\NN_0\rightarrow \RR^+_0$ be a non-decreasing function. Suppose there exists an integer $k$ such that for any integer $n \geq 2$ it holds that
		$$\ln^k(\kgv\range n)\leq f(\kgv\range n) ,$$
		then $\ln^k \prec f$
	\end{lemma}
	\begin{proof}
		Let $x>2$ be arbitrary, then there exists some $n\geq 2$ such that $$\kgv\range n\leq x\leq \kgv\range{n+1}\leq(n+1)\kgv\range n.$$
		We then have the following sequence of inequalities:\begin{align*}
			f(x)&\geq f(\kgv\range n)\\
			&\geq \ln^k(\kgv\range n)\\
			&=\frac{1}{3^k}\left(\ln\left(\kgv\range n\right)+\ln\left(\kgv\range n\right)+\ln\left(\kgv\range n\right)\right)^k\\
			&\geq \frac{1}{3^k}\left(\ln(\kgv\range n)+\ln\left(\frac{n+1}{2}\right)+\ln(2)\right)^k\\
			&\geq \frac{1}{3^k}\ln^k\left(\kgv\range n (n+1)\right)\\
			&\geq \frac{1}{3^k}\ln^k(\kgv\range{n+1})\\
			&\geq \frac{1}{3^k}\ln^k(x).
		\end{align*}
		This demonstrates what we needed to show.
	\end{proof}
	\begin{lemma}
		\label{prop:comparegeometric}
		Let $f,g:\NN_0\rightarrow\RR^+_0$ be non-decreasing functions. Suppose that there exists a constant $C>1$ and an increasing sequence of natural numbers $(n_i)_{i\in \NN}$ such that for every natural number $m$, there exists some $i \in \NN$ with $n_i\in[m,Cm]$.  If $f(n_i)\leq g(n_i)$ for all $n_i$ in this sequence, then $f\prec g$.
	\end{lemma}
	\begin{proof}
		Let $m$ be arbitrary and take $i \in \NN$ such that $n_i \in [m,Cm]$. Then we have the following inequalities:
		\begin{align*}
			g(Cm)&\geq g(n_i)\\
			&\geq f(n_i)\\
			&\geq f(m)
		\end{align*}
		Since this holds for all values of $m$, we thus have $g\prec g$.
	\end{proof}
	The lemma above states that lower bounds of the conjugacy separability function need not be checked everywhere, but just for a sufficiently dense set of values. Clearly, it suffices to demonstrate bounds on a geometric sequence. Also the set of primes on an arithmetic progression is sufficiently dense:
	
	\begin{lemma}\label{prop:psubgeom}
		Let $a$ and $b$ be coprime positive integers and $p_i$ be the increasing sequence of primes in the set $a+b\NN$. Then there exists some constant $C\in\NN$ such that for every natural number $m$, there lies a prime number in the interval $[m,Cm]$.
	\end{lemma}
	\begin{proof}
		This is a consequence of the prime number theorem for arithmetic progressions as found and proven in \cite{delavallepoussain1899Recherces}.
	\end{proof}
	
	A class of groups for which the conjugacy separability is known, is the class of finitely generated abelian groups, see \cite{bou2010quantifying}.
	\begin{proposition}
		Let $H$ be an infinite finitely generated abelian group, then $\conj_H(n) \simeq \ln(n)$
	\end{proposition}
	This is in fact a consequence of the following lemma, which is a consequence of the prime number theorem for arithmetic progressions, and which we will also need in our proofs.
	\begin{lemma}
		\label{prop:pnondevider}
		Let $a$ and $b$ be coprime integers. There exists a constant $C_a > 0$ such that for every strictly positive integer $n$, there exists some prime number $p$ such that $p \nmid n$ and $p =  b\mod a$. This prime can be chosen such that $p\leq C_a+C_a\ln(n)$.
	\end{lemma}
	
	\subsection {Representations of finite groups}
	In this section, we recall the basics about representations of finite groups over an integral domain. From now on, a group $G$ will be finite unless explicitly otherwise stated and $R$ is a general integral domain although $R$ will be either $\ZZ$, $\QQ$, $\CC$, $\ZZ[\zeta_n]$, $\QQ[\zeta_n]$ or $\frac{\ZZ}{p\ZZ}$ in all examples and applications. Here $\zeta_n \in \mathbb{C}$ is a primitive $n$-th root of unity and $p$ is always a prime.
		
	\begin{definition}
	A \textbf{representation} of $G$ over $R$ is a group morphism $\rho:G\rightarrow \mathrm{Aut}_R(M)$ with $M$ an $R$-module. We will also denote this as $\rho:G\acts M$ in short. A \textbf{subrepresentation} of $\rho:G\acts M$ is a $R$-submodule $N \subset M$ that is invariant under $\rho$, meaning that for any $g\in G$ it holds that $\rho(g)(N)=N$. 
	\end{definition}
	Note that for a subrepresentation it suffices to check that $\rho(g)(N) \subset N$ for all $g \in G$. If $N$ is a subrepresentation of $\rho:G\acts M$, then $\rho$ also induces a map $\rho_N:G\acts N$. This induced map $\rho_N$ is a representation of $G$ on $N$, and sometimes we will call this the subrepresentation as well. We give two elementary examples.
	
	\begin{example}
	Every representation $\rho: G \acts M$ has the trivial subrepresentation $N = \{0\}$.
	\end{example}
	\begin{example}
		Let $\rho:G\acts M$ be a representation on a free $\ZZ$-module. For every $n \in \NN$ the submodule $nM$ is a subrepresentation of $\rho$. Indeed if $v\in M$, then $\rho(g)(nv)=n\rho(g)(v)$ and as such is $\rho(g)(nv)$ also an element of $nM$. 
	\end{example}
	
	We recall certain notions about representations.
	\begin{definition}
		Let $\rho:G\acts M$ and $\rho':G\acts N$ be two representations over some integral domain $R$. A morphism $\varphi:M\rightarrow N$ of $R$-modules is called a\textbf{ $G$-morphism} if it respects the action, that is such that for every $g\in G$, it holds that $\rho'(g)\circ \varphi=\varphi \circ \rho(g)$.
	\end{definition}
	
	\begin{definition}
		A representation over a field $F$ is \textbf{irreducible} if its only subrepresentations are the trivial representation and itself. 
	\end{definition}
	For integral domains that are not a field however, we need a slightly more general definition. To introduce it we first need to recall the tensor product.
	\begin{definition}
		Let $R$ be an integral domain containing $R' \subset R$ as a subring. Let $\rho:G\acts M$ be a representation on a free $R'$-module $M$, then we denote with $\rho\otimes_{R'}R$ the representation $\rho\otimes_{R'}R:G\acts M\otimes_{R'}R$, such that ${(\rho\otimes_{R'}R)(g)(m\otimes r)=\left(\rho(g)m\right)\otimes r}$.
	\end{definition}
This allows us to define irreducibility for general integral domains.
	\begin{definition}
		We call a representation over an integral domain $R$ \textbf{irreducible} if it is irreducible over the field of fractions $Q$ of $R$. More precisely, $\rho:G\acts M$ is irreducible if $G\acts M\otimes Q$ is irreducible as a $Q$-representation.
	\end{definition}
	\begin{lemma}\label{prop:iredIfSubspaceFiniteIndex:lemma}
		Let $R$ is an integral domain of characteristic $0$ that is finitely generated as a $\ZZ$-module and $\rho$ be a representation on a finitely generated free $R$-module. The representation $\rho$ is irreducible if and only if all subrepresentations $N$ are either $0$ or of finite index in $M$ (as abelian groups).
	\end{lemma}
	\begin{proof}
		First assume all subrepresentations are either trivial or of finite index. Let $M$ be a free $R$-module and let $\rho:G\acts M$ be a representation. Let $N_Q$ be a non-trivial subrepresentation of $M\otimes_R Q$. Then take $n\in N_Q$ a non-zero element of $N$, which is by definition of the form $n=\frac{m}{r}$ with $m \in M$ and $r \in R$. It follows that $m=rn\in N_Q\cap M$. It thus also follows that $M\cap N_Q$ is a non-trivial subrepresentation on $R$, and thus of finite index in $M$. As $N_Q$ is a representation on a field of characteristic $0$, it follows that $N_Q=M\otimes_R Q$.
		
		For the converse let $M$ be once again a free $R$-module. Let $\rho:G\acts M$ be a representation and let $N$ be a non-trivial subrepresentation, then we will show that $N$ has finite index in $M$. Now $N\otimes Q$ is a subrepresentation of $M\otimes Q$ which is also non-trivial. As $M\otimes Q$ is irreducible, it follows that $M\otimes Q=N\otimes Q$. In particular, for every element $m\in M$, there exist some elements $\lambda_m\in R$ and $n_m\in N$ such that $m=\frac{n_m}{\lambda_m}$. Consider $\betha$ a free basis for $M$ as $R$-module and let $\lambda$ be $\prod_{m\in \beta}(\lambda_m)$, then $\lambda M$ must be contained in $N$. 
		
		To conclude we demonstrate that $\lambda M$ is of finite index in $M$. First as $R$ is an integral domain, it follows that left multiplication by $\lambda$ is injective. If we see $R$ as a $\ZZ$-module, then left multiplication by $\lambda$ is a morphism of $\ZZ$-modules. In particular, as $M$ is a free $R$-module, multiplication by $\lambda$ is an injective morphism of $\ZZ$-modules. As $M$ is of finite dimension over $R$, and $R$ is finitely generated as a group, $M$ is also of finite dimension over $\ZZ$. As $\lambda M$ is a $\ZZ$-submodule of the same dimension as $M$, $\lambda M$ must have finite index.
	\end{proof}
	
	\begin{example}
		Let $M$ be the module $\ZZ^2$ and let $G$ be the group $\frac{\ZZ}{4\ZZ}=\set{\overline0,\overline1,\overline2,\overline3}$. Define $\rho$ via $\rho(\overline1) = \begin{pmatrix} 0 & -1 \\ 1 & 0 \end{pmatrix}$, so $\rho(\overline1)$ is a rotation over $\frac{\pi}{2}$ radians in the plain. This action is irreducible over $\ZZ$. Indeed, suppose that $N$ is a non-trivial subrepresentation, containing the non-trivial element $(x,y)$. As $N$ is $\rho$-invariant, then $\rho(\overline1)(x,y)=(-y,x)$ must lie in $N$ as well. As $(x,y)$ and $(-y,x)$ are linearly independent, it follows that the subgroup they generate must have rank at least $2$, which is the same rank as $M$. It follows that $M_1$ must be of finite index in $\ZZ^2$. Since this holds for all subrepresentations, it must hold that $\rho$ is irreducible.
	\end{example}
	\begin{definition}
		Let $M$ be an $R$-module and $N$ be a submodule, then the \textbf{radical} $\sqrt{N}$ is the module containing the following elements$$
		\sqrt{N}=\set{m\in M\mid \exists r\in R \setminus \{ 0 \} : rm\in N}
		$$
	\end{definition}
	When $M$ is a $\ZZ$-module and thus an abelian group, then the radical is also known as the isolator subgroup.
	\begin{lemma}\label{prop:radicalIsRepresentation:lemma}
		Let $M$ be an $R$-module and $\rho:G\acts M$ be a representation. For every subrepresentation $N \subset M$, it holds that $\sqrt{N}$ is also a subrepresentation of $M$.
	\end{lemma}
	\begin{proof}
		Take $m \in \sqrt{N}$ arbitrary, then there exists $r\in R \setminus \{0\}, n\in N$ such that $rm=n$. For any $g\in G$, we have $r\rho(g)m=\rho(g)rm=\rho(g)n\in N$. It follows that $\rho(g)m$ also belongs to $\sqrt{N}$
	\end{proof}
	
	To end this section, we discuss the behaviour of a representation on a field, after tensoring with a field extension.
	\begin{definition}
		Let $F\subset E$ be a finite degree Galois extension, with corresponding Galois group $\Gal(E,F)$. For every vector space $M$ over $F$, the action of $\Gal(E,F)$ on $M\otimes_FE$ is defined as $\sigma(m\otimes x)=m\otimes(\sigma(x))$ for any $\sigma \in \Gal(E,F)$ and $m\otimes x$ in $M\otimes_FE$. 
	\end{definition}
	\begin{loanthm}
	[{\cite[Theorem 9.21]{isaacs2006character}}]	\label{thm:shur}
		Let $F\subset E$ be a finite degree Galois extension of fields of characteristic $0$ and $\rho:G\acts V$ an irreducible representation on a vector space $V$ over $F$. Then there exists an integer $m$, called the Shur index, such that the representation $\rho\otimes_FE$ can be written as $\displaystyle \bigoplus_{i\in I}\rho_i^m$, where $\rho_i = \rho_{V_i}$ for $V_i \subset V$ are non-isomorphic irreducible subrepresentations over $E$. Moreover, $m \mid \abs {G}$ and the $V_i$ can be chosen such that they are conjugate under the action of $\Gal(E,F)$ on the set of irreducible subrepresentations.
	\end{loanthm}
	
	The latter means that for every $i \in I$ it holds that $\sigma(V_i) = V_j$ for some $j \in I$, and that moreover this action is transitive. However, finding irreducible representation can be done over a much smaller ring than $\CC$, by the following theorem. Here, for any finite group $G$, we define $\zeta_G$ as a primitive $\abs{G}$'th root of unity.
	
	\begin{loanthm}[{Brauer\cite[Theorem 10.3.]{isaacs2006character}}]\label{prop:iredoversubfield:theorem}
		Let $G$ be a finite group. Let $\rho:G\acts M$ be an irreducible representation on the free $\ZZG$-module $M$. Then $\rho\otimes\CC$, the representation on the space $M\otimes\CC$, is also irreducible.
	\end{loanthm}

	\subsection{Characters for representations}
	In this section, we introduce several results from representation theory using characters. Some of them are standard and will be provided without a proof, we refer to \cite{isaacs2006character} for more details.
	
	Given a representation $\rho$ of a group $G$ over a subfield $F \subset \CC$, then we can associate to this representation the character $\chi_\rho:G\rightarrow \CC$. This maps every group element $g\in G$, to the trace of its action $\chi_\rho(g)=\trace \left(\rho(g)\right)$. This character is invariant under conjugation, that is for $g,h\in G$, we have $\chi_\rho(g)=\chi_\rho(hgh\inv)$. This map is also linear in the sense that if $\rho_1$ and $\rho_2$ are representations, then $\chi_{\rho_1\oplus\rho_2}=\chi_{\rho_1}+\chi_{\rho_2}$ where $+$ denotes pointwise addition.

	On the space of characters, we define something resembling a Hermitian product.
	\begin{definition}
		Let $\chi_1$ and $\chi_2$ be characters over a subfield $F$ of the complex numbers, then we define $$(\chi_1,\chi_2)_G = \sum_{g\in G}\chi_1(g)\overline{\chi_2(g)}.$$\end{definition}
	\noindent Note that we do not divide by $\abs{G}$ as is usual in character theory. The map $(\cdot,\cdot)$ is sesquilinear and positive definite. Furthermore the irreducible representations of $G$ over $F$ form an orthogonal set. If $F$ is algebraically closed and $\rho$ is irreducible, then it holds that $(\chi_\rho,\chi_\rho)=\abs{G}$.
	
	To each character and representation, we now associate a map.
	\begin{definition}
		Let $F$ be a number field and $\chi$ a character over $F$. For every representation $\rho:G\acts M$ we define the map 	$\varpi_{\chi,M}:M\rightarrow M$ as 
		$$
\varpi_{\chi,M}(v)= \sum_{g\in G}\overline{\chi(g)}\rho(g)v
		$$
	\end{definition}
\noindent	Contrary to a slightly more common definition, we do not multiply these maps by $\frac{\chi(1)}{\abs{g}}$. Notice that this map is defined in terms of a module $M$, but we will sometimes omit the domain $M$ when no confusion is possible.
	
As the following lemmas will make clear, these maps are multiples of projections on irreducible subrepresentations.
	\begin{lemma}\label{prop:pigmap}
	For every $\chi$ and $G \acts M$, the map $\varpi_\chi: M \to M$ is a $G$-map.
	\end{lemma}
	\begin{proof}
		Let $\rho:G\acts M$ be a representation.
		It is obvious that $\varpi_{\chi,M}$ is linear as a linear combination of linear maps. We compute the image of $\rho(h)v$ for some vector $v\in M$ and some group element $h\in G$.
		\begin{align*}
			\varpi_{\chi,M}(\rho(h)v)&=\sum_{g\in G}\overline{\chi(g)}\rho(g)\left(\rho(h)v\right)\\
			&= \rho(h)\sum_{g\in G}\overline{\chi(g)} \rho(h\inv g h) v \\
			&= \rho(h)\sum_{g\in G}\overline{\chi(g)} \rho(g)v\\
			&=\rho(h)\varpi_{\chi,M}(v)
		\end{align*}
	\end{proof}
	\begin{notation}
		If $\rho'$ is a representation with corresponding character $\chi_{\rho^\prime}$, then we also write $\varpi_{\rho'}$ for $\varpi_{\chi_{\rho'}}$.
	\end{notation}
	\begin{lemma}
		\label{prop:repindicator}
		Let $G$ be a group and $\rho_1:G\acts M_1, \rho_2:G\acts M_2$ two non-isomorphic irreducible $\CC$-representations. Then we have the following:
		$$
		\varpi_{\rho_1,M_2}=0
		$$
		and $$
		\varpi_{\rho_1,M_1}=c_{\rho_1}\id,
		$$
	 where $c_{\rho_1} = \frac{\abs{G}}{\dim(M_1)} \in \mathbb{Z}$, and hence divides $\abs{G}$.
	\end{lemma}
	\begin{proof}
		By the previous lemma, $\varpi_{\rho_1,M_i}$ is a $G$-map from an irreducible representation over an algebraically closed field to itself, and thus corresponds to multiplication with some constant $c_i$. We thus have $\trace(\varpi_{\rho_1,M_i})=c_i\dim(M_i)$. 
		
		On the other hand, $\varpi_{\rho_1,M_i}$ is also a linear combination of linear maps, so the trace of $\varpi_{\rho_1,M_i}$ is that same linear combination of the traces of those maps. Hence the trace of $\varpi_{\rho_1,M_i}$ is equal to
		$$
		\sum_{g\in G}\overline{\chi_{\rho_1}(g)}\trace(\rho_i(g))
		$$
		which is nothing else then $(\chi_1,\chi_i)$. So if $i=2$ we get $c_2 = 0$ or equivalently that $\varpi_{\rho_1,M_2}$ is multiplication by $0$. On the other hand, if $i=1$ then $ (\chi_1,\chi_1) = \abs{G}$, and thus $c_{1} = \frac{\abs{G}}{\dim(M_i)}$. It is well-known that if $M_1$ is irreducible, the dimension of $M_1$ must divide $\abs{G}$ or thus $c_1$ is an integer.
	\end{proof}
	
	We generalize this lemma to work over $\QQ$ and $\ZZ$ as well.
	\begin{lemma}
		\label{prop:repindicatorZ}
		Let $G$ be a group and $\rho_1:G\acts M_1, \rho_2:G\acts M_2$ be two non-isomorphic irreducible representations over $\QQ$. Then we have the following:
		$$\varpi_{\chi_{\rho_1},M_2}=0
		$$
		and
		$$\varpi_{\chi_{\rho_1},M_1}=c_{\rho_1}\id_{M_1},
		$$
		where $c_{\rho_1} \in \ZZ$ divides $\abs{G}^2$
	\end{lemma}
	\begin{proof}
		Consider the complex representation $\rho_1\otimes\CC$ and its decomposition $$\rho_1\otimes\CC=\bigoplus \rho_{1,i}^m$$ into irreducible subrepresentations from Theorem \ref{thm:shur}. 
	Here we assume that $\rho_{1,i}$ and $\rho_{1,j}$ are non-isomorphic when $i$ and $j$ are distinct. Denote by $M_{1,i}$ a subrepresentation on which $\rho_{1,i}$ acts.
	The character of $\rho_1$ is then given by $m\sum\chi_{\rho_{1,i}}$.
	
	The map $\varpi_{\chi_{\rho_1},M_i}$ is linear in the sense that $\varpi_{\chi_1+\chi_2}=\varpi_{\chi_1}+\varpi_{\chi_2}$, hence $\varpi_{\chi_{\rho_1},M_i}$ can thus be rewritten as $m\sum\varpi_{\chi_{\rho_{1,i}}}$.
	If we consider this map on any of the representations $\rho_{1,j}$, we obtain $$
	\varpi_{\chi_{\rho_1,M_{1,j}}}=\sum m\varpi_{\chi_{\rho_{1,i},M_{1,j}}},
	$$
	which then by Lemma \ref{prop:repindicator} reduces to $\varpi_{\rho_1,M_1} = m\varpi_{\chi_{\rho_{1,j},M_{1,j}}}=mc_{\rho_{1,j}}\id_{M_{1,j}}$. As $c_{\rho_{1,j}} = \frac{\abs{G}}{\dim(\rho_{1,j})}$, it is independent of $j$, showing that the identity holds on all of $M_1$. A similar argument holds on $M_2$. As the integer $m$ must divide the order of the group, we get that $c_{\rho_1} = c_{\rho_{1,j}} m$ divides $\abs{G}^2$. 
	\end{proof}
	The main use of this lemma will be the fact that if you start with an integral representation, that the map $\varpi_{\chi_{\rho_1},M_1}$ will still be a map over the integers. The following remark shows that we can do even better.
	
	\begin{remark}\label{prop:repindicatorZstrong:remark}
		If we consider $\varpi_{\frac{\chi_{\rho_1}}{m}}$, with $m$ the Schur index of $\rho_1$, then we claim that $\frac{\chi_{\rho_1}}{m}=\sum \chi_{\rho_{1,i}}$ will still be integer-valued. Note that $\chi_{\rho_1}$ is the character of a representation on $\QQ$ and thus $\frac{\chi_{\rho_1}}{m}$ has values in $\QQ$. Furthermore, the eigenvalues of the maps $\rho_{1,i}$ are all roots of unity, and thus algebraic integers. We conclude that $\frac{\chi_{\rho_1}}{m}=\sum \chi_{\rho_{1,i}}$ is valued in the algebraic integers, leading to the claim.
	
	As a conclusion, if $\rho:G\acts M'$ is a finitely generated representation on $\ZZ$, then the maps $\varpi_{\frac{\chi_{\rho_1}}{m}}$  restricts to a map $M'\rightarrow M'$. Doing this we obtain a map that on subrepresentations isomorphic to $\rho_1$ acts like multiplication by some constant dividing $\abs{G}$.
	\end{remark}

	The maps $\varpi_\chi$ are well-behaved with respect to $G$-maps.
	\begin{lemma}\label{prop:indicommut}
		Let $\rho_{1}: G\acts M_1$ and $\rho_{2}:G\acts M_2$ be two representations. Take $G'$ a subgroup of $G$ and consider $\rho'_1,\rho'_2$ the representations $\rho_1\mid_{G'}$ and $\rho_2\mid_{G'}$ respectively. For any $G$-map $\varphi: M_1\rightarrow M_2$ and any character $\chi$ of $G^\prime$, it holds that $\varphi\comp\varpi_{\chi,M_1}=\varpi_{\chi,M_2}\comp\varphi$.
	\end{lemma}
	\begin{proof}
		We apply both maps on some $v\in M_1$:
		\begin{align*}
			&\varphi\comp\varpi_{\chi,M_1}v\\
			=&\varphi\left(\sum_{g\in G'}\overline{\chi(g)}\rho'_1(g)v\right)\\
			=&\sum_{g\in G'}\overline{\chi(g)}\varphi(\rho'_1(g)v)\\
			=&\sum_{g\in G'}\overline{\chi(g)}\rho'_2(g)(\varphi(v))\\
			=&\varpi_{\chi,M_2}\comp\varphi(v).
		\end{align*}
	\end{proof}

	\section{Strategy for the Main Result}\label{sec:result}
	\numberwithin{lemma}{section}
The main result of this paper determines the conjugacy separability function for virtually abelian groups. 
	\begin{theorem}\label{prop:maintheorem:theorem}
		Let $H$ be a finitely generated virtually abelian group of Hirsch length $h$, then there exists some integer $0 \leq k \leq h$ such that $$
		\conj_H(n)\simeq \ln^k(n).
		$$
	\end{theorem}
	The Hirsch length of $H$ is the rank of a finite index torsion-free abelian subgroup. There will be two explicit formulas for the integer $k$, and most of the remainder of this section is devoted to describing these. As we will show, each formula has its own advantage for either the lower or upper bound, and by comparing both constants we find that they are in fact equal.

	First of all, not every virtually abelian group is a semi-direct product of a torsion-free abelian group with a finite one. The following result using group cohomology however shows that these groups can be embedded into a semi-direct product of a very specific form.
	\begin{proposition}\label{prop:virtabEmbedsInSplitVirtab:lemma}
		Let $H$ be a finitely generated virtually abelian group.
		Then there exists some representation $\rho:G\acts M$ on a torsion-free abelian group $M$ such that $H$ is a subgroup of $M\rtimes_\rho G$, extending $\abs{G}M$ by $G$.	\end{proposition}
	The last statement of the lemma means that $H\cap M=\abs{G}M$ and that the projection of $H$ to $G$ is surjective.
	\begin{proof}	
		As $H$ is finitely generated virtually abelian, it fits in a short exact sequence$$
		\begin{tikzcd}
			0\ar[r]&M'\ar[r]&H\ar[r]&G\ar[r]&1
		\end{tikzcd}
		$$
		where $M'$ is free abelian and $G$ is finite. Write $M = \frac{M'}{\abs{G}} \subset M^\prime \otimes \QQ$ for the free abelian group with elements of the form $\set{\frac{m}{\abs{G}}\in M'\otimes\QQ \Big\vert m\in M'}$ and $\rho$ for the linear extension of the representation of $G$ on $M'$ induced by the extension $H$. We will show the existence of an injective morphism $\varphi$ such that the following commutes.
				\[\begin{tikzcd}
			0 & M' & H & G & 1 \\
			0 & {M} & {M\rtimes_{\rho} G} & G & 1
			\arrow[from=1-1, to=1-2]
			\arrow["\iota", from=1-2, to=1-3]
			\arrow["\pi", from=1-3, to=1-4]
			\arrow[from=1-4, to=1-5]
			\arrow[from=2-1, to=2-2]
			\arrow[from=2-2, to=2-3]
			\arrow[from=2-3, to=2-4]
			\arrow[from=2-4, to=2-5]
			\arrow[from=1-2, to=2-2]
			\arrow["\varphi", from=1-3, to=2-3]
			\arrow["{=}", from=1-4, to=2-4]
		\end{tikzcd}.\]
		
		By \cite[Lemma 6.6.5]{weibel1994introduction}, the extension 
		\[\begin{tikzcd}
			0 & M' & H & G & 1
			\arrow[from=1-1, to=1-2]
			\arrow[from=1-2, to=1-3]
			\arrow[from=1-3, to=1-4]
			\arrow[from=1-4, to=1-5]
		\end{tikzcd}\]
		is completely determined by a certain representation $\rho: G \acts M^\prime$ and by one of its factor sets, namely a map $$[\_,\_]:G\times G\rightarrow M':(g,g')\mapsto [g,g']=\sigma(g)\sigma(g')\sigma(gg')\inv$$ where $\sigma: G \to H$ is some map on sets such that $\pi \circ \sigma = \id_G$.
In this case, the group is isomorphic to the set $M'\times G$ where multiplication is given by:$$
(m,g)*(m',g')=(m+\rho(g)m'+[g,g'],gg').
$$

	 Similarly, by applying the embedding of $M'$ into $M$, we also obtain a factor set on $M$. One easily checks that the extended representation and factor set define a new group $H'$, and that $H$ embeds naturally into this group. It thus remains to be shown that the group $H'$ is a semi-direct product. Through multiplication by $\abs{G}$ we obtain an isomorphism between $M'$ and $M=\frac{M'}{\abs{G}}$. This isomorphism preserves the representation $\rho$, but it maps the extension of $[\_,\_]$ to the factor set $\abs{G}[\_,\_]$.
	 
	 By \cite[Theorem 6.6.7]{weibel1994introduction}, these factor sets are precisely the normalized 2-cocycles. Also by \cite[Theorem 6.5.8]{weibel1994introduction}, $H^2(G,A)$ is annihilated by $\abs{G}$. It follows that $\abs{G}[\_,\_]$ is not only a normalized $2$-cocycle, but also a coboundary. It thus follows by \cite[Theorem 6.6.3]{weibel1994introduction}, that it is up to a change of section, equivalent to the trivial factor set and thus that $H'$ is a semi-direct product, completing the proof.
	\end{proof}
	\noindent The previous proposition shows that we may restrict ourselves to subgroups $H \le M\rtimes_\rho G$ of semi-direct products that extend $\abs{G}M$ by $G$ where $\rho: G \acts M$ is some representation of a finite group onto a torsion-free finitely generated abelian group. For a discussion about the uniqueness of these representation and the relation to the holonomy representation for crystallographic groups, we refer to \cite[Section 3]{dere2023residual}. 
	
\begin{convention}
From now on, unless stated otherwise, let $\rho:G\acts M$ be a representation as above of a finite group on a free $\ZZ$-module. The subgroup $H <M\rtimes_\rho G$ is always an extension of $\abs{G}M$ by $G$. We fix for every $h\in G$ an element $v_h\in M$ such that $(v_h,h)\in H$. 
\end{convention}
\noindent We will study conjugation in $H$ using the following two submodules. 
	
	\begin{definition}
		Let $\rho:G\acts M$ be a representation of a finite group on a free $\ZZ$-module. For every $g\in G$, we define $W_g(M) = \Ima (\id_M-\rho(g)) \subset M$ and $V_g(M)=\sqrt[M]{W_g(M)}$.
	\end{definition}

	\begin{lemma}\label{prop:conjugatesinrepresentation:lemma}
	With notations as above, the elements $(v_1,g),(v_2,g)\in H$ are conjugate if and only if there exists some $h\in C(g)$ such that $v_1-\rho(h)v_2-(v_h-\rho(g)v_h) \in \abs{G} W_g(M)$.
		
	\end{lemma}
	\begin{proof}
	The elements $(v_1,g), (v_2,g) \in H$ are conjugate if there exists some element $(-\tilde v,h)\in H$ such that$$
		(v_1,g)=(-\tilde v,h)(v_2,g)(\tilde v,h)^{-1}.
		$$
A short computations shows that this is equivalent to \begin{align*}
			(v_1+\tilde v-\rho(g)\tilde v,g)&=(\rho(h)v_2,hgh\inv).
		\end{align*}
			For these to be equal, we first need that $hgh\inv=g$ or thus that $h\in C(g)$, and secondly that $v_1+\tilde v-\rho(g)\tilde v=\rho(h)v_2$. Note that by construction every element in $H$ is of the form $(v,0)(v_h,h)$ for some $v \in \abs{G} M$, and that any element of this form lies in $H$. In particular, $(-\tilde v,h) \in H$ if and only if $\tilde{v} + v_h \in \abs{G}M$, leading to the claimed result.
	\end{proof}
	
	If $H$ is a group as above, and $N \subset \abs{G} M$ a subrepresentation, then $N$ always forms a normal subgroup in $H$, and we can take the quotient $H/N$. These are the finite quotients we will consider to study conjugacy separability.
		
	\begin{corollary}\label{prop:conjugatesinrepresentationquotient:corollary}
	Let $N$ be a subrepresentation of $\abs{G}M$, then $(v_1,g),(v_2,g)\in H/N$ are conjugate if and only if there exists some $\tilde v$ in $\abs{G}M$, and some $h\in C(g)$ such that $v_1-\rho(h)v_2 - (v_h-\rho(g)v_h)\in N + \abs{G} W_g(M)$.
	\end{corollary}
	\begin{proof}
		This follows by applying Lemma \ref{prop:conjugatesinrepresentation:lemma} on $M/N$.
	\end{proof}
	
In the following definition, we look for pairs $(v_1,v_2)$ for which the conjugacy classes can be separated in smalls quotients of the group $H$, namely by the normal subgroup $\abs{G}^3 M$.
	\begin{definition}
	With notations as above, take a $g \in G, v_1,v_2\in M$ and $h\in C(g)$. Then we call $(v_1,v_2)$ \textbf{strongly $g$-unsolvable} in $h$ if $v_1-\rho(h)v_2-v_h+\rho(g)v_h\notin \abs{G}W_g(M)+\abs{G}^3M$. If no confusion is possible, we just say that $(v_1,v_2)$ is \textbf{strongly unsolvable}.
	\end{definition}
		Notice that the above definition does not depend on the choice of $v_h$, indeed replacing $v_h$ with $v_h'$, adds a factor $(v_h-v_h')-\rho(g)(v_h-v_h')$ which is an element of $\abs{G}W_g(M)$. Furthermore notice that $\rho(g)v_h-v_h$ itself belongs to $V_g(M)$.

	\begin{example}
		Let $G=\frac{\ZZ}{3\ZZ}\oplus\frac{\ZZ}{2\ZZ}\oplus\frac{\ZZ}{2\ZZ}$ with generators $g_1, g_2$ and $g_3$ of the three components, and let $M=\ZZ^{3\times2}$. An element $m$ of $M$ is denoted by a $3$ by $2$ matrix$$
		\begin{pmatrix}
			m_1&m_2\\m_3&m_4\\m_5&m_6
		\end{pmatrix}
		$$
		The action of $G$ on $M$ is such that$$
		\rho(g_1)m=\begin{pmatrix}
			m_5&m_6\\m_1&m_2\\m_3&m_4
		\end{pmatrix},
\\
		\rho(g_2)m=\begin{pmatrix}
			m_2&m_1\\m_4&m_3\\m_6&m_5
		\end{pmatrix},
	\\
\rho(g_3)m=-m.$$
		
		Take $g=g_1g_3\in G$, then as $G$ is abelian, we know that $G=Z(g)$.
		Applying the map $\id_M-\rho(g)$ on the basis of $M$, we obtain that $W_g(M)$ is generated by$$
		\set{\begin{pmatrix}1&0\\1&0\\0&0\end{pmatrix},\begin{pmatrix}0&0\\1&0\\1&0\end{pmatrix},\begin{pmatrix}1&0\\0&0\\1&0\end{pmatrix},\begin{pmatrix}0&1\\0&1\\0&0\end{pmatrix},\begin{pmatrix}0&0\\0&1\\0&1\end{pmatrix},\begin{pmatrix}0&1\\0&0\\0&1\end{pmatrix}}
		$$
	A short computation shows that this is a subspace of index $4$, and hence that $V_g(M)=M$.
		
		Let $H$ be the subgroup of $M\rtimes G$ generated by $\abs{G}M\rtimes1$ and $0\rtimes G$. In this case we can choose $v_h=0$ for any choice of $h\in G$.
		Consider then $v_1$ and $v_2$ as follows:$$
		v_1=\begin{pmatrix}1&0\\0&0\\0&0\end{pmatrix}
		$$
		and$$
		v_2\begin{pmatrix}0&1\\0&0\\0&0\end{pmatrix}.
		$$
		Then $(v_1,v_2)$ is $g$-strongly unsolvable in $h=e$.
		Indeed $v_1-\rho(h)v_2-(v_h-\rho(g)v_h)=v_1-v_2$ which is given by $$
		\begin{pmatrix}1&-1\\0&0\\0&0\end{pmatrix}
		$$
		and this is not an element of $W_g(M)+12^3M$. On the other hand $(v_1,v_2)$ is not $g$-strongly unsolvable in $h = g_2g_3$ as $v_1-\rho(h)v_2-(v_h-\rho(g)v_h)$ in this case is given by the matrix$$
		\begin{pmatrix}2&0\\0&0\\0&0\end{pmatrix}
		$$
		which is an element of $W_g(M)$.
		
	\end{example}

To study quotients to separate elements that are not strongly unsolvable, we will look at the $\QQ$-irreducible representations of our representation $\rho:G\acts M$.  The representation $M\otimes \QQ$ splits into subrepresentations $\tilde N_1\oplus\tilde N_2\oplus\cdots\oplus\tilde N_l$, where 
we assume that all irreducible subrepresentations of $\tilde N_i$ are isomorphic and if $\tilde N_i$ and $\tilde N_j$ have an isomorphic irreducible subrepresentation, then the indices must be equal, i.e.~$i=j$. It is easy to see that this decomposition is unique.
	\begin{definition}
		Let $\rho:G\acts M$ be a representation on a free $\ZZ$ module and let $\tilde N_1\oplus\tilde N_2\oplus\cdots\oplus\tilde N_l$ be the decomposition as above.
		Define $\pi_i:M\rightarrow\tilde N_i$ to be the natural projection maps and define $M_i=\pi_i(M)\subset \tilde N_i$.
		We call $M_1\oplus M_2\oplus\cdots \oplus M_l$ the \textbf{square hull} of $\rho$. When no confusion is possible, we also say that $M_1\oplus M_2\oplus\cdots \oplus M_l$ is the square hull of $M$.
	\end{definition}

	\begin{example}
		Let $G=\frac{\ZZ}{2\ZZ}=\set{\overline0,\overline1}$ and let $M=\ZZ^2$. Let $\rho:G\acts M$ be such that $$\rho(\overline1)=\begin{pmatrix}0&1\\1&0\end{pmatrix}.$$
		As a $\QQ$ vector space, the irreducible subrepresentations of $M\otimes\QQ$ are spanned by $(1,1)$ and $(1,-1)$. The square hull is here thus given $\left\langle\left(\frac{1}{2},\frac{1}{2}\right)\right\rangle\oplus\left\langle\left(\frac{1}{2},-\frac{1}{2}\right)\right\rangle$.
	\end{example}

	The square hull $M_1\oplus M_2\oplus\cdots \oplus M_l$ contains $M$ as a subgroup and in a lot of cases, this inclusion is strict. However $M$ can not be a lot smaller then $M_1\oplus M_2\oplus\cdots \oplus M_l$.

	\begin{lemma}\label{prop:squareHullAlmostIncluded:lemma}
		Let $M'=M_1\oplus\cdots\oplus M_l$ be the square hull of $M$, then ${\abs{G}M'\subset M\subset M'}$
	\end{lemma}
	\begin{proof}
		Let $\rho_i$ be the irreducible component of $\QQ\otimes M_i$. Let $\varpi_{\frac{\chi_{\rho_i}}{m}}$ be the map from Remark \ref{prop:repindicatorZstrong:remark}. The image of this map is contained in $M$. By the same Remark, the image of this map contains $\abs{G}M_i$. Since $\abs{G}M_i$ is a subspace of $M$ for every index $i$, it follows that $\abs{G}(M_1\oplus\cdots\oplus M_l)=\abs{G}M'\subset M$.
	\end{proof}
	The square hull plays an important role in computing the conjugacy separability. The main strategy is is to compute the conjugacy separability functions on each of the groups $M_i\rtimes G$, and reducing the conjugacy separability function on $M\rtimes G$ to the product of some of these functions. The way the conjugates of an element in an abelian by finite group look, depend heavily on their projection on the finite group. Therefore in a lot of what follows, we start by fixing an element $g\in G$, the finite group. 
	
We introduce $4$ notions that measure whether or not we can separate conjugacy classes in small quotients in the $i$-th component or not. Although quite technical at this moment, the different cases will become clear in the proofs of our main results. 
	\begin{definition}
		Let $\rho:G\acts M$ be a representation of a finite group on a free $\ZZ$-module with square hull $M_1\oplus M_2 \oplus\cdots \oplus M_l$. Fix a positive integer $m$ and an element $g\in G$.
		
		For any $i\in\range{l}$, $h\in C(g)$ and $v_1,v_2\in M$, we say that \begin{itemize}
			\item the pair	$(v_1,v_2)$ \textbf{$g$-vanishes} in $(i,h)$ if $\pi_i(v_1-\rho(h)v_2)\in V_g(M_i)$; 
			\item otherwise, so if $\pi_i(v_1-\rho(h)v_2)\notin V_g(M_i)$, we say that $(v_1,v_2)$ is \textbf{$g$-weakly unsolvable} in $(i,h)$. We distinguish two different cases:
			\begin{itemize}
			\item$(v_1,v_2)$ is \textbf{$g,m$-locally unsolvable} in $(i,h)$ if ${\pi_i(v_1-\rho(h)v_2)\notin \abs{G}^mM_i+V_g(M_i)}$;\\
			\item$(v_1,v_2)$ is \textbf{$g,m$-globally unsolvable} in $(i,h)$ otherwise. 
\end{itemize}
			 \end{itemize}
	\end{definition}
	 \noindent  Notice that in above definitions, we do not talk about a specific subgroup $H$, or elements $v_h$. This is a consequence of the fact that $v_h-\rho(g)v_h\in V_g(M)$ and thus that adding this term would not change the equations. The bigger we take the integer $m$, the weaker the condition of locally unsolvable becomes, and thus how stronger the condition of globally unsolvable.
	
The idea of the previous definition is as follows. In the first case, the orbits are identical, and as such can not be separated in a finite quotient. If $(v_1,v_2)$ is locally unsolvable in $(i,h)$, then a quotient of exponent $m$ will suffice to separate conjugacy classes. However, in the globally unsolvable case, such an exponent will not suffice, but the orbits can be separated in some (potentially large) quotient that we will describe later on. 
	\begin{example}
		Let $G=\frac{\ZZ}{2\ZZ}\times\frac{\ZZ}{2\ZZ}$ act on $M=\ZZ^3$ as follows:$$
		\rho:G\rightarrow \mathrm{Aut}(M):(\overline{g_1},\overline{g_2})\mapsto\begin{pmatrix}
			(-1)^{g_1}&0&0\\
			0&(-1)^{g_2}&0\\
			0&0&(-1)^{g_1+g_2}
		\end{pmatrix}.
		$$
		Let $H$ be the subgroup $\abs{G}M\rtimes G\subset M\rtimes G$. Take $v_1=(1,\abs{G}^5,1)$ and $v_2=(-1,-\abs{G}^5,-1) = -v_1$. If $g=\overline{(0,0)}$ then we have that $W_g(M)=0$ and thus also $V_g(M)=0$. The decomposition $\ZZ\oplus\ZZ\oplus\ZZ$ is already one into irreducible factors and is thus the square hull of $G\acts M$. The projection maps onto these subrepresentations are the natural ones.
		
		We show that $(v_1,v_2)$ $g$-vanishes in $(1,h)$ where $h=(\overline{1,0})$. Indeed $v_1-\rho(h)v_2=(0,2\abs{G}^5,0)$ and thus is $\pi_1(v_1-\rho(h)v_2)=0$.
		In $(2,h)$ the pair $(v_1,v_2)$ is $g,5$-globally unsolvable, as $\pi_2(v_1-\rho(h)v_2)=2\abs{G}^5$ is an element of $\abs{G}^5\ZZ$ but not an element of $V_g(M_2)=0$. It thus also follows that $(v_1,v_2)$ is $g$-weakly unsolvable in $(2,h)$. Lastly $(v_1,v_2)$ is locally $g,5$-unsolvable in $(3,h')$ where $h'=(\overline{0,0})$. Indeed $\pi_3(v_1-\rho(h')v_3)=2\notin\abs{G}^5M_3+V_g(M_3)=\abs{G}^5M_3$. Again it follows that $(v_1,v_2)$ is $g$-weakly unsolvable.
	\end{example}
	
	We introduce $2$ similar notions.
	\begin{definition}
	Fix an element $g\in G$, then we say $K\subset\range{l}$ \textbf{$g,m$-admits} $(v_1,v_2)\in M\times M$ if for every $h\in C(g)$, either\begin{itemize}
			\item $(v_1,v_2)$ is $g$-strongly unsolvable with respect to $h$;\\
			\item there exists some $i\in\range{l}$ such that $(v_1,v_2)$ is $g,m$-locally unsolvable in $(i,h)$;\\
			\item there exists some $i\in K$ such that $(v_1,v_2)$ is $g,m$-globally unsolvable.
		\end{itemize}
	\end{definition} 
	
	The central idea behind above definition is (as will be proven in Lemma \ref{prop:conjMultipleCoppies:lemma}) that if $M_i$ is as above, then it is as easy to separate one element of $M_i$ from $V_g(M_i)$ in a finite quotient, as it is to separate multiple elements from that space. In the above, $K$ represents certain subrepresentations $M_i$ such that every element $v_1-\rho(h)v_2$ is either distinct from $V_g(M)$ after projection on one of these $M_i$ corresponding to $K$ (globally-unsolvable), or if it does not lie in $K$, it can be recognised in some quotient of uniformly bounded size (locally or strongly unsolvable).
	
	\begin{definition}
	Fix an element $g\in G$ and let $v_1,v_2,k_1,k_2\in M$. We say $K\subset\range{l}$ \textbf{$g$-admits} $(v_1,v_2,k_1,k_2)$ if for every $h\in C(g)$, one of the following holds:\begin{itemize}
			\item $(k_1,k_2)$ is $g$-strongly unsolvable in $h$;\\
			\item there exists some $i\in\range{l}$ such that $(k_1,k_2)$ is $g$-weakly unsolvable in $(i,h)$;\\
			\item there exists some $i\in K$ such that $(v_1,v_2)$ is $g$-weakly unsolvable in $(i,h)$.
		\end{itemize}
	\end{definition}
	
The intuition for the $K$ in this definition is similar as above, with the difference that the elements $(k_1,k_2)$ now capture the uniformly bounded case. By increasing $K$, the conditions become weaker because of the following lemma.
	
	\begin{lemma}\label{prop:admitancehereditary:lemma}
		Let $K\subset K'\subset\range{l}$. If $g \in G$ and $v_1,v_2,k_1,k_2 \in M$ are such that $K$ $g$-admits$(v_1,v_2,k_1,k_2)$ then $K'$ $g$-admits $(v_1,v_2,k_1,k_2)$ as well. Similarly if $m > 0$ and $K$ $g,m$-admits $(v_1,v_2)$ then so does $K'$.
	\end{lemma}
	Both of the above definitions talk about a subset $K$ of $\range{l}$. To each of these subsets we associate a positive integer.
	\begin{definition}
		Let $\rho:G\acts M$ be a representation of a finite group on a free $\ZZ$-module with square hull $M_1\oplus M_2 \oplus\cdots \oplus M_l$ and $d_i$ the dimension of one of the irreducible subrepresentations of $M_i\otimes \CC$. For every $K\subset\range l$ the \textbf{dimension} of $K$ is given by:$$
		\dim K =\sum_{i\in K}d_i.
		$$
	\end{definition}
	\begin{remark}
		By Theorem \ref{thm:shur}, in the above we may replace the dimension of one of the irreducible subrepresentations of $M_i\otimes\CC$, by the dimension of one of the irreducible subrepresentation of $M_i\otimes\ZZG$.
	\end{remark}
	With this we can state a refinement of Theorem \ref{prop:maintheorem:theorem}.
	
	\begin{theorem}\label{prop:conjugacyseparabilittysplitvirtuallyabelian:theorem}
	With notations as above, we have that $\conj_H(n)\simeq \ln^{\mf k_1}$ and there exists some integer $m_H$ such that the following are equal\begin{itemize}
			\item $\mf k_1$;\\
			\item $\mf k_2=\max_{g\in G}\max_{\substack{(v_1,g)\in H\\(v_2,g)\in H}}\min\set{\dim K\mid K \text{ }g,m_H\mathrm{-admits} (v_1,v_2)}$;\\
			\item $\mf k_3=\max_{g\in G}\max_{\substack{(k_1,g),(k_2,g)\in H\\v_1,v_2\in M}}\min\set{\dim K\mid K \text{ }g\mathrm{-admits} (v_1,v_2,k_1,k_2)}$.
		\end{itemize}
	\end{theorem}
	\begin{remark}
		In the above theorem, when the minimum is taken over the empty set, then we take the minimum to be equal to $0$ (instead of $\infty$ which would be more common)
	\end{remark}
	
	By the previous remark, $\mf k_3$ is an integer that always exists and thus Theorem \ref{prop:maintheorem:theorem} follows. In the remainder of the paper we will prove Theorem \ref{prop:conjugacyseparabilittysplitvirtuallyabelian:theorem}. First we will demonstrate in Section \ref{sec:2->3} that for some specific value $m_H>3$ (and thus also all higher values) $\mf k_2\leq\mf k_3$. Then we will demonstrate in Section \ref{sec:3->1} that $\conj_H(n)\succ \ln^{\mf k_3}(n)$ (notice that these statements are both independent of $m$). Lastly in Section \ref{sec:1->2} we will demonstrate that whenever $m>3$, then ${\conj_H(n)\prec\ln^{\mf k_2}(n)}$. The theorem then follows as $\ln^{k_1}(n)\simeq\ln^{k_2}(n)$ if and only if $k_1=k_2$.
	\begin{corollary}
		Let $H$ be a finitely generated virtually abelian group with irreducible holonomy representation, then $\conj_G(n)\simeq\mathrm{Rf}_G(n)$.
	\end{corollary}
	\begin{proof}
		Let $G\acts M$ be the holonomy representation of $H$. By Lemma \ref{prop:virtabEmbedsInSplitVirtab:lemma} we can realise $H$ as a subgroup of $M\rtimes G$ extending $\abs{G}M$ with $G$. As $G\acts M$ is irreducible, the square hull decomposition is just $M$. Notice that for any choice of $g,v_1,v_2,k_1,k_2$, the set $$
		\set{\dim K\mid K \text{ }g\mathrm{-admits} (v_1,v_2,k_1,k_2)}
		$$ can contain no values larger then $\dim(\{1\})$, which is of course the dimension $d$ of one of the $\CC$-irreducible subrepresentation of $M\otimes \CC$.
		It follows that $$\mf k=\max_{g\in G}\max_{\substack{(k_1,g),(k_2,g)\in H\\v_1,v_2\in M}}\min\set{\dim K\mid K \text{ }g\mathrm{-admits} (v_1,v_2,k_1,k_2)}$$ is also bounded above by $d$ and thus that $\conj_H(n)\prec\ln(n)^d$.
		From \cite{dere2023residual} we have that $\mathrm{Rf}_H(n)\simeq\ln(n)^d$ and we thus have that $\conj_H(n)\prec\mathrm{Rf}_H(n)$. The result follows as $\conj_H(n)\succ\mathrm{Rf}_H(n)$ is always the case.
	\end{proof}

\label{sec:extra}

At the end of this section, we give some examples illustrating the  the bounds of Theorem \ref{prop:conjugacyseparabilittysplitvirtuallyabelian:theorem}. Notice that $M\otimes \CC$ splits in irreducible representations and these can be grouped up to Galois conjugation. Write $\mc M$ for this partition. Then $\conj_H(n)\prec \ln^k(n)$ where $k=\sum_{[M]\in\mc M}\dim(M)$. This bound however need not be sharp, as we illustrate with the following example.

	\begin{example}\label{prop:naiveupperbound:remark}
	Consider the group $H_0=\ZZ\rtimes\frac{\ZZ}{2\ZZ}$ where the second group acts non-trivially on the first. Then the group $H_0\oplus H_0$ has the same conjugacy separability function as $H_0$, being $\ln(n)$. However, $H_0\oplus H_0$ can as a virtually abelian group be written as $\ZZ^2\rtimes \frac{\ZZ}{2\ZZ}^2$ where the action is given by$$
		\rho:\frac{\ZZ}{2\ZZ}^2\rightarrow \mathrm{Aut}(\ZZ^2)
		$$
		such that $$\rho(1,0)=\begin{pmatrix}-1&0\\0&1
		\end{pmatrix}$$
		
		and$$\rho(0,1)=\begin{pmatrix}1&0\\0&-1
		\end{pmatrix}.$$
		
		Both $\ZZ\oplus\set{0}$ and $\set{0}\oplus \ZZ$ are irreducible subrepresentations, and as they have rank one, they remain irreducible over $\CC$. As $(1,0)$ acts non-trivially on the first, but trivially on the latter these representations are non-isomorphic over both $\QQ$ and $\CC$, even after Galois conjugation. Hence the bound described above only gives the non-exact value $\ln^2(n)$.
	\end{example}
		The conjugacy separability function may in some cases be strictly larger then the residual finiteness function as described in \cite{dere2023residual}.
	\begin{example}
	
		Let $G=\frac{\ZZ}{2\ZZ}\times\frac{\ZZ}{2\ZZ}$ act on $M=\ZZ^3$ as follows:$$
		\rho:G\rightarrow \mathrm{Aut}(M):(\overline{g_1},\overline{g_2})\mapsto\begin{pmatrix}
			(-1)^{g_1}&0&0\\
			0&(-1)^{g_2}&0\\
			0&0&(-1)^{g_1+g_2}
		\end{pmatrix}.
		$$
		
		Note that $M\rtimes G$ can also be realized as a subgroup of itself extending $\abs{G}M$ by $G$, with the same representation of the group $G$. To avoid confusion we denote with $H = \abs{G}M \rtimes G$ the subgroup of the bigger group $H_0=M\rtimes G$.
		The square hull of $M$ is in this case just $\ZZ\oplus\ZZ\oplus\ZZ$.
		
		Let $g = e$ and let $k_1=k_2=0$, then both $(k_1,g)$ and $(k_2,g)$ are trivial in $H_0$ and thus elements of $H$. It is immediate that $(k_1,k_2)$ $g$-vanishes in $(i,h)$ for all $i\in\left\{1, 2, 3\right\}$ and for all $h\in C(g)=G$. It also follows that $(k_1,k_2)$ is strongly unsolvable for no element $h\in C(g)$.
		Let $v_1=(1,1,1)\in M$ and let $v_2=(-1,-1,-1)\in M$. As $g$ acts trivially, both $W_g(M)$ and $V_g(M)$ are just the trivial module. Quick computations show that $(v_1,v_2)$ vanishes in the following cases.\\
		\begin{figure*}[!h]
		\centering
		\begin{tabular}{l|l l l}
			&1&2&3\\
			\hline
			$(\overline0,\overline0)$&u&u&u\\
			$(\overline0,\overline1)$&u&v&v\\
			$(\overline1,\overline0)$&v&u&v\\
			$(\overline1,\overline1)$&v&v&u\\
		\end{tabular}
	\end{figure*}\\
		A $v$ in column $i$ and row $h$ indicates that $(v_1,v_2)$ $g$-vanishes in $(i,h)$, similarly a $u$ indicates that $(v_1,v_2)$ is $g$-weakly unsolvable in $(i,h)$.
		We thus see that $\{1,2,3\}$ admits $(g,v_1,v_2,k_1,k_2)$. It is also easy to check that this is not the case for any proper subset of $\{1,2,3\}$, for instance for $K=\{1,2\}$, $(v_1,v_2)$ vanishes both in $(1,h)$ and $(2,h)$ where $h=(\overline1,\overline1)$. By the remark before Example \ref{prop:naiveupperbound:remark} it thus follows that $\conj_H(n)\prec\ln^{\dim(\{1,2,3\})}(n)$ and by Theorem \ref{prop:conjugacyseparabilittysplitvirtuallyabelian:theorem} these two are even equivalent. All three components of the square hull are irreducible over $\CC$ thus $\conj_H(n)\simeq\ln^3(n)$.
		
		As $G$ is abelian, all irreducible subrepresentation of $M\otimes \CC$ are $1$-dimensional. By the main result of \cite{dere2023residual} it thus follows that the residual finiteness function is given by $\ln(n)$. We thus have found a virtually-abelian group where the conjugacy separability function is strictly larger than the residual finiteness function. In a similar fashion, one can construct examples where the difference becomes arbitrarily large.
	\end{example}
	\section{Comparing the constants}\label{sec:2->3}
	In this section we will demonstrate that the constants $\mf k_2$ and $\mf k_3$ of our main result satisfy $\mf k_2\leq\mf k_3$. The central idea is to decompose the pair $(v_1,v_2)$ in the definition of $\mf k_2$ as a pair $(k_1+v_1',k_2+v_2')$ such that $k_1$ and $k_2$ capture information about the $m_H$-local unsolvability and strong unsolvability of $(v_1,v_2)$ and such that $v_1'$ and $v_2'$ capture the information about $m_H$-global unsolvability.
	
In order to find such a decomposition, the constant $m_H$ must be sufficiently large, mainly to prevent interference between strong and global unsolvability. First we have the following definition.
	
	\begin{definition}
		Let $M$ be a $\ZZ$-module and $v\in M$ a non-zero element. We define the \textbf{index} $I \in \mathbb{N}$ of $v$ in $M$ as the maximal positive integer such that there exists some element $v_0\in M$ with $v=Iv_0$. In this case we call the corresponding element $v_0 \in M$ the \textbf{radical} of $v$ in $M$. A submodule $M^\prime \subset M$ is called radical if $M^\prime$ contains the radical of all its elements.
	\end{definition}
The size of $m_H$ is a consequence of the following three lemmas, which allow to rewrite element satisfying additional properties. 
	
	\begin{lemma}\label{prop:fixPowerDecomposition:lemma}
		Let $M$ be a free $\ZZ$-module containing radical submodules $M_1$ and $M_2$ of finite dimension. For any $N \in \mathbb{N}$, there exists some $n \in \mathbb{N}$ such that for every $v_1\in M_1, v_2\in M_2$ and $m \in \mathbb{N}$ with $v_1-v_2\in N^{m+n}M$, there exists some $v_2'\in M_1\cap M_2$ such that $v_1-v_2'\in N^mM$.
	\end{lemma}
	\begin{proof}
		Take $N \in \mathbb{N}$ arbitrary. Define $M_1', M_2' \subset M$ such that $M_1=M_1'\oplus(M_1\cap M_2)$, $M_2=M_2'\oplus(M_1\cap M_2)$ and thus for $M_0=M_1+M_2$ it holds that $M_0=M_1'\oplus M_2'\oplus (M_1\cap M_2)$. Let $O$ be the exponent of the finite group $\frac{\sqrt{M_0}}{M_0}$ and define $n$ as the highest power for which some prime factor of $N$ appears in $O$.
		
		The difference $v_1-v_2$ is an element of $M_0$. Let $I$ be the index of $v_1-v_2$ in $M_0$ and $w$ be its radical in $M_0$. We know that $v_1-v_2\in N^{m+n}M$ and thus can be written as $N^{m+n}v_0$ for some $v_0\in M$. As $v_1-v_2$ belongs to $M_0$, we know that $v_0\in\sqrt{M_0}$. In particular $O'v_0\in M_0$ for some minimal integer $O'$ which satisfies $O'\mid O$.
		Notice that $w$ is also the radical of $O'v_0$ and thus $O^\prime v_0$ must be a multiple of $w$. We thus have $N^{m+n}\mid O'I$ and in particular $N^{m+n}\mid OI$. By the definition of $n$, we thus have that $N^m\mid I$, and in particular, $v_1-v_2$ can be written as $N^mv$ for some $v\in M_0$.
		
		As $M_0=M_1'\oplus M_2'\oplus (M_1\cap M_2)$, $v$ can in a unique way be written as $w_1+w_2+w_3$, where $w_1 \in M_1^\prime, w_2 \in M_2^\prime$ and $w_3\in M_1\cap M_2$ respectively. It follows that $v_1-v_2-N^mw_2=N^m(w_1+w_3)$ is an element of $M_1$, and in particular $v_2+N^mw_2$ is an element of $M_1$. As both $v_2$ and $w_2$ are elements of $M_2$, we also have that $v_2+N^mw_2\in M_2$. Lastly, as both $v_1-v_2$ and $N^mw_2$ belong to $N^mM$, we have that $v_1-v_2-N^mw_2$ also belongs to $N^mM$. From these it follows that $v_2'=v_2+N^mw_2$ satisfies the conditions.
	\end{proof}
	\begin{lemma}\label{prop:inverseImageRetainsPower:lemma}
		Let $M$ be a free $\ZZ$-module, and let $\varphi$ be an endomorphism on $M$. For any $N \in \mathbb{N}$, there exists some $n \in \mathbb{N}$ such that for every $m \in \mathbb{N}$ and $v\in M$ such that $\varphi(v)\in N^{m+n}M$ holds that $v$ can be written as $v'+k$ such that $k\in\ker(\varphi)$ and such that $v'\in N^mM$
	\end{lemma}
	\begin{proof}
	Take $N \in \mathbb{N}$ arbitrary. Let $O$ the exponent of the group $\frac{\sqrt{\varphi(M)}}{\varphi(M)}$ and $n$ be the highest power with which some prime factor of $N$ appears in $O$. Assume that $v \in M$ such that $\varphi(v) = N^{m+n}w$ for some $w \in M$. Similarly as in the previous lemma, we have that $N^n w\in\varphi(M)$, so there exists $v'' \in M$ such that $N^nw=\varphi(v'')$. Then $\varphi(N^nv'')=\varphi(v)$ and thus it follows that $v=N^nv''+k$ where $k$ belongs to the kernel of $\varphi$.		
	\end{proof}
	
	\begin{lemma}\label{prop:representationsAreScalable:lemma}
		Let $G\acts M$ be a representation of a finite group over a free $\ZZ$-module $M$ and let $N$ be a positive integer. There exists some constant $n$ such that for any $m >0$ and any $v\in M$ with corresponding subgroup ${\mc H=\set{h\in G\mid v=\rho(h)v\mod N^{n+m}M}}$, there exists some $k\in M$ such that $\rho(h)k=k$ for all $h\in\mc H$ and $v-k\in N^{m}M$.
	\end{lemma}
	\begin{proof}
		In fact, we will show that for any subgroup $\mc H<G$ there exists some constant $n_{\mc H}$ such that if for all $m>0$ and $v\in M$ with $v-\rho(h)v\in N^{n_{\mc H}+m}M$ for all $h\in {\mc H}$, then there exists some $k\in M$ such that $k=\rho(h)k$ for all $h\in {\mc H}$ and $v-k\in N^{m}M$. This statement is equivalent as $G$ only has a finite number of subgroups. We demonstrate this fact via induction on the number of generators of ${\mc H}$. 
		
		For the basis step with ${\mc H}$ the trivial subgroup, the result holds true for $n_{\mc H}=0$ and $v=k$. Suppose that the result holds true for any subgroup $\mc H_0$ with at most $i-1$ generators.
		Let ${\mc H}$ be a subgroup generated by the set $\{ h_1,h_2,\cdots h_i\}$ and let ${\mc H}_0$ be the subgroup of ${\mc H}$ generated by $\{h_1,h_2,\cdots,h_{i-1}\}$. Let $n_1$ be the integer from Lemma \ref{prop:fixPowerDecomposition:lemma} where $M_1$ and $M_2$ are given by \begin{align*}
			M_1&=\bigcap_{h\in {\mc H}_0}\ker(\id_M-\rho(h)) \\ M_2&=\ker(\id_M-\rho(h_i)).\end{align*} Let $n_2$ be the constant from Lemma \ref{prop:inverseImageRetainsPower:lemma} where $\varphi=(\id_M-\rho(h_i))$. We claim the result holds for $n_{\mc H}= n_1 + n_2 +n_{{\mc H}_0}$.
		
		Let $m>0$ be arbitrary and let $v\in M$ be such that $v-\rho(h)v\in N^{n_{\mc H}+ m }M$ for all $h\in {\mc H}$. In particular, we have that $v-\rho(h)v\in N^{n_1+n_2+n_{{\mc H}_0}+m}M$ for all $h\in {\mc H}_0$. By the induction hypothesis there exist some $k\in M$ such that $k=\rho(h)k$ for all $h\in {\mc H}_0$ and such that $v-k\in N^{n_1 + n_2 +m }M.$ By assumption, we know that $v-\rho(h_i)v\in N^{n_{\mc H}+m}M$, and we also have that $v-k-\rho(h_i)(v-k)\in N^{n_1+n_2+m}M.$ Combining these two equations it follows that $k-\rho(h_i)k\in N^{n_1+n_2+m}M$. By Lemma \ref{prop:inverseImageRetainsPower:lemma}, we have that $k$ can be written as $k=k'+v'$ where $\rho(h_{i})k'=k'$ and $v'\in N^{n_1 + m}M$. Since $k\in\bigcap_{h\in {\mc H}_0}\ker(\id_M-\rho(h))$, we can apply Lemma \ref{prop:fixPowerDecomposition:lemma} on $k-k'\in N^{n_1 + m}M$ and obtain that there exists some $k''\in M_1 \cap M_2$ such that $k-k''\in N^{m}$. Notice that as $k''=\rho(h)k''$ on a generating set of ${\mc H}=\langle {\mc H}_0,h_i \rangle$, it is invariant under $\rho(h)$ for every $h\in {\mc H}$. The result follows as $v-k''\in N^{n_1 + n_2 + m}M+ N^{m}M=N^{m}M$.
	\end{proof}

	The last lemma will be used to decompose a pair $(v_1,v_2)$ as ${(v_1'+k_1,v_2'+k_2)}$ such that $v_1'$ and $v_2'$ belong to $\abs{G}^3 M$, and as such they will not influence what happens separability at a small scale (which we called strong unsolvability). Notice that in the above lemma, the constant $n$ is not unique, however if the result holds for $n$, then so does it for higher values.
	\begin{notation}\label{not:globalconstant}
		Let $G\acts M$ be a representation with square hull $M_1\oplus\cdots\oplus M_l$. Let $n_{i,g}$ be the minimal constant for which Lemma \ref{prop:representationsAreScalable:lemma} holds for the representation $C(g)\acts M_i/V_g(M_i)$ and the constant $N=\abs{G}$. Then we define a constant $m_H=5 + \max \left\{n_{i,g} \mid i \in \{1, \ldots, l\}, g \in G \right\}$ for the group $H$ with notations as before.
	\end{notation}
We added the notation $H$ to make clear that we will apply this fixed constant $m_H$ further for the virtually abelian groups $H$ under consideration. However, one can define this constant for any representation $G \acts M$. We will demonstrate that the lower bound of Theorem \ref{prop:conjugacyseparabilittysplitvirtuallyabelian:theorem} holds for $m=m_H$ and the upper bound for any sufficiently large value of $m$ (here $m\geq3$ suffices). First, we need some more information about the subgroups $V_g(M)$ and $W_g(M)$. 
	
	\begin{lemma}\label{prop:WIndexbound:lemma}
		Let $G$ be a finite group with $g \in Z(G)$, and let $G\acts M$ be an irreducible representation on a free $\ZZ$-module $M$ such that $g$ acts non-trivially. Then $\abs G M\subset W_g(M)$.
	\end{lemma}
	\begin{proof}
	Write $T = \rho(g)$ and let $n$ be the order of $T$. First we will show that $$
		\Phi_n(T)=0
		$$
		where $\Phi_i$ is the $i$'th cyclotomic polynomial.	Notice that whenever $i\mid n$ but $i\neq n$, then $\Phi_i(T) \neq 0$ does not vanish, as this would imply that $T^i-\id_M = 0$ and thus that $T$ has order $i$. As $g$ is central, the linear map $T$ is a $G$-map, and thus it follows that $\Phi_i(T)$ is also a $G$-map. The kernel of $\Phi_i(T)$ is thus either finite index in $M$, or trivial. As $M$ is torsion-free, the kernel must be radical, and thus by the previous the first option does not occur, showing that the kernel must be $0$.	It follows in particular that 
		$$\prod_{\substack{i\mid n\\i\neq n}}\Phi_i(T)$$
		also has trivial kernel. Notice that $$
		\prod_{\substack{i\mid n\\i\neq n}}\Phi_i(T) \cdot \Phi_n(T)=T^n-\id_M=0,
		$$
	which implies that $\Phi_n(T)$ must have a non-trivial kernel. Since $\Phi_n(T)$ is a $G$-map, this implies that $\Phi_n(T)=0$.
		
	Write $R= \frac{\ZZ[Y]}{(\Phi_n(Y))}$, then using the previous we find that there exists a morphism of $\ZZ$-algebras
	$$
		\psi:R \rightarrow \ZZ[T]
		$$ by mapping $Y$ to $T$, and where the right hand side is regarded as a subring of the endomorphism algebra on $M$.
		In $R[X]$ it holds that $$
		(X-Y)(X-Y^2)\cdots(X-Y^{n-1})=X^{n-1}+X^{n-2}+\cdots+1.
		$$
		Evaluating this polynomial in $1$ then gives$$
		(1-Y)(1-Y^2)\cdots(1-Y^{n-1})=n.
		$$
		If we apply the morphism $\psi$ to the previous equality, we obtain$$
		n=(1-T)(1-T^2)\cdots(1-T^{n-1})
		$$
		In particular it follows that $n M \subset \Ima (1-T) = W_g(M)$, and thus the result follows from the fact that $n\big\vert\abs{\rho(G)}\big\vert\abs{G}$.
	\end{proof}
	
	We can now find a radical subspace derived from $W_g(M)$.
		\begin{lemma}\label{prop:WRadicalInSubgroup:lemma}
		Let $\rho:G\acts M$ be a representation of a finite group on a free $\ZZ$-module. For any $g\in G$, the subspace $W_g(M)\cap\abs{G}^2M$ is radical in $\abs{G}^2M$.
	\end{lemma}
	\begin{proof}
		In the above statement, we may assume that $g$ is central in $G$, as if this is not the case we may replace $G$ by $C(g)$. Indeed, this does not change the action of $g$ on $M$, and hence neither $W_g(M)$. Moreover, $\abs{C(g)}$ is a divisor of $\abs{G}$.
		
		Let $M_1\oplus M_2\oplus\cdots\oplus M_n\supset M$ be the square hull of $M$ and consider $\abs{G}\pi_i: M \to M$ be the $G$-map taking $(m_1,m_2,\cdots m_n)$ to $(0,0,\cdots \abs{G}m_i,0,\cdots,0)$. By Remark \ref{prop:repindicatorZstrong:remark} this is precisely some integer multiple of the map $\varpi_{\chi_i}$, and can thus indeed be seen as a map from $M$ to $M$.
		Now for any $w\in V_g(M)$, we have that $\abs{G}\pi_i(w)\in\abs{G}\pi_i(V_g(M))\subset V_g(M_i)$. Furthermore $\abs{G}w=\sum_{i=1}^l\abs{G}\pi_i(w)$. 
		
	It suffices to show that for any $w \in V_g(M)$, it holds that $\abs{G}^2 w \in W_g(M)$. We claim it even suffices to show that whenever $w\in V_g(M_i)$, then $\abs{G}w\in W_g(M_i)$. Indeed if this is the case, and $w\in V_g(M)$, then we would also have $\abs{G}\pi_i(w)\in V_g(M_i)$ and thus that $\abs{G}^2\pi_i(w)\in W_g(M)$ or thus $\abs{G}^2w=\sum_{i=1}^l \abs{G}^2\pi_i(w)\in W_g(M).$
		
		For every subspace $M_i$, there are two possibilities, namely either $g$ acts trivially, or it does not. In the first case, all expression of the form $v-\rho(g)v$ vanish, and as such $W_g(M_i)$ and $V_g(M_i)$ are trivial. As indeed $\abs{G}0=0$, the statement in this case holds true. So it suffices to consider the subspaces $M_i$ where $\rho(g)$ is non-trivial. As we know that all irreducible subrepresentations of $M_i \otimes \QQ$ are isomorphic, the map $\rho(g)$ has the same order on each of them. It thus follows by the Lemma \ref{prop:WIndexbound:lemma} that $\abs{G} M_i \subset W_g(M_i)$, which was exactly the claim.
	\end{proof}
	Using the case distinction of the previous proof, we find also a description of $V_g(M)$.
	
	\begin{lemma}\label{prop:structOfVgM:lemma}
		Let $G\acts M$ be a representation of a finite group on a free abelian $\ZZ$-module and $g \in Z(G)$. If we write $M \otimes \QQ$ as decomposition into $\QQ$-irreducible representations $N_1\oplus N_2\oplus\cdots \oplus N_m$ and $\mc I=\set{i\in\range{m} \Big\vert \rho(g)\mid_{N_i}\neq\id_M}$, then $$V_g(M)=M\cap\bigoplus_{i\in\mc I}N_i.$$
	\end{lemma}
	\begin{proof}
		Notice that $\rho(g)$ is a $G$-map, and as such $\Ima\left(\id_M-\rho(g)\right)$ is a subrepresentation.
		For $i\in\range m$, suppose there exists some $v\in N_i\without 0$ such that $v\neq\rho(g)v$, then $(\id_M-\rho(g))N_i$ is a non-trivial subrepresentation and as such all of $N_i$. Conversely if $\rho(g)$ acts trivially on $v\in N_i\without 0$, then the kernel of $(\id_M-\rho(g))$ is non-trivial, and as such acts $\rho(g)$ trivially on all of $N_i$.
		
		We show the equality via two inclusions. First assume that $v\in M\cap\bigoplus_{i\in\mc I}N_i$, then by the previous $v$ can be written as $\sum_{i\in I}v_i-\rho(g)v_i$ where $v_i\in N_i$. Let $N >0$ be such that $Nv_i\in M$ whenever $i\in \mc I$, then ${Nv=\sum_{i\in\mc I}Nv_i-\rho(g)Nv_i\in W_g(M)}$ and thus $v\in V_g(M)$. Conversely, for any $v=\sum_{i=1}^l v_i$ with $v_i\in N_i$, we have that $$v-\rho(g)v=\sum_{i\notin\mc I}v-v + \sum_{i\in\mc I}v-\rho(g)v\in\bigoplus_{i\in\mc I}N_i.$$ In particular, $W_g(M) \subset M \cap \bigoplus_{i \in \mc I} N_i$ and thus the other inclusion holds as well because the right hand side is radical.
	\end{proof}
	\begin{lemma}\label{prop:modOutVg:lemma}
		Let $G\acts M$ be a representation of a finite group on a free abelian $\ZZ$-module and $g \in Z(G)$. Then there exists some morphism ${\varpi_{V_g(M)}:M\rightarrow M}$ with kernel $V_g(M)$ and image $\ker(\id_M-\rho(g))$. Furthermore, on $\Ima(\varpi_{V_g(M)})$ this map acts like multiplication by $\abs{G}$.
	\end{lemma}
	\begin{proof}
		Let $M_1\oplus\cdots\oplus M_l$ be the square hull of $M$. By Remark \ref{prop:repindicatorZstrong:remark} there are projection maps $\abs{G}\pi_i:M\rightarrow M_i\cap M$ that act like multiplication by $\abs{G}$ on their image. Let $\mc I=\set{i\in\range{l}\mid \rho(g)|_{M_i} \neq \id_M}$ as before, then by Lemma \ref{prop:structOfVgM:lemma}, $$\varpi_{V_g(M)}=\sum_{i\in\mc I^c}\abs{G}\pi_i$$ has the desired properties.
	\end{proof}
	
	The following lemma shows where $\abs{G}^3$ is needed in the definition of strongly unsolvable.

	\begin{lemma}\label{prop:unsolvableInProjection:lemma}
		Let $\rho:G\acts M$ be a representation of a finite group on a free $\ZZ$-module with square hull $M_1\oplus\cdots\oplus M_l$. Take $g\in G$ and $v\in M$ with $v\notin \abs{G}W_g(M)$, then either $v\notin \abs{G}W_g(M)+\abs{G}^3M$, or there exists some $i$, such that $\pi_i(v)\notin V_g(M_i)$.
	\end{lemma}
	\begin{proof}
	We first assume $\pi_i(v)\in V_g(M_i)$ for all $i$, then there exist positive integers $n_i > 0$ such that $n_i\pi_i(v)\in W_g(M_i)$. Then $\abs{G}n_1n_2\cdots n_lv\in W_g(M)$ and thus $v\in V_g(M)$.
		Suppose additionally that ${v\in \abs{G}W_g(M)+\abs{G}^3M}$, then $v$ can be written as $v=\abs{G}w+\abs{G}^3m$ with $w\in W_g(M)$ and $m\in M$. As both $v$ and $w$ belong to $V_g(M)$, $m$ must too. It follows by Lemma \ref{prop:WRadicalInSubgroup:lemma} that  $\abs{G}^2m$ belongs to $W_g(M)$, or in conclusion that $v = \abs{G} w + \abs{G}^3 m$ belongs to $\abs{G}W_g(M)$, which proves the result.
	\end{proof}

We finally use this to show that there exists at least one set $K$ satisfying the third condition of Theorem \ref{prop:conjugacyseparabilittysplitvirtuallyabelian:theorem}.

	\begin{lemma}\label{prop:ExistenceAdmittingSubset:Lemma}
		Let $m$ be a strictly positive integer and let $\rho:G\acts M$ be a representation of a finite group on a free $\ZZ$-module with square hull $M_1\oplus\cdots\oplus M_l$. Fix $g$ and let $v_1,v_2\in M$ be such that for any $h\in C(g)$ holds that ${(v_1-\rho(h)v_2-v_h+\rho(g)v_h)\notin \abs{G}W_g(M)}$. Then there exists some $K\subset\range{l}$ $g,m$-admitting $(v_1,v_2)$.
	\end{lemma}
	\begin{proof}
		This is immediate from Lemma \ref{prop:unsolvableInProjection:lemma}.
	\end{proof}
	\begin{proposition}
		Let $\rho:G\acts M$ be a representation of a finite group on a free $\ZZ$-module. Let $H<M\rtimes G$ be an extension of $\abs{G}M$ by $G$.
		Let $$\mf k_2=\max_{g\in G}\max_{\substack{(v_1,g)\in H\\(v_2,g)\in H}}\min\set{\dim K\mid K \text{ }g,m_H\mathrm{-admits} (v_1,v_2)}$$
		where $m_H$ is as in Notation \ref{not:globalconstant} and let $$\mf k_3=\max_{g\in G}\max_{\substack{(k_1,g),(k_2,g)\in H\\v_1,v_2\in M}}\min\set{\dim K\mid K \text{ }g\mathrm{-admits} (v_1,v_2,k_1,k_2)}.$$
		Then $\mf k_2\leq \mf k_3$.		
	\end{proposition}
	\begin{proof}
	 Write $M^\prime = {M_1\oplus M_2 \oplus\cdots \oplus M_l}$ for the square hull of $M$ and let $g\in G$ and $(v_1,g),(v_2,g)\in H$ such that $\mf k_2=\min\set{\dim K\mid K \text{ }g,m_H\mathrm{-admits} (v_1,v_2)}$. As $\mf k_3 \geq 0$, we may assume in the proof that $\mf k_2>0$. Hence we know that $\set{\dim K\mid K \text{ }g,m_H\mathrm{-admits} (v_1,v_2)}$ is not empty, and thus by Lemma \ref{prop:admitancehereditary:lemma} it has $\range{l}$ as an element. We will demonstrate that there exist some $(k_1,g),(k_2,g)\in H$ and $v'_1,v'_2\in M$ such that $\range{l}$ $g\mathrm{-admits} (v'_1,v'_2,k_1,k_2) $ and such that if $K\subset\range{l}$ $g$-admits $(v'_1,v'_2,k_1,k_2)$, then $K$ also $g,m_H$-admits $(v_1,v_2)$. This indeed shows that $\mf k_3$ is an upper bound for $\mf k_2$.
		
		To find these elements, we will first construct $k_{1i},k_{2i},v_{1i},v_{2i}$ for all $i\in\range{l}$ with $v_{ji}\in M\cap M_i$, $\sum_{i=1}^l k_{ji}+\abs{G}^3v_{ji}=v_j$ and such that\begin{itemize}
			\item  $(v_1,v_2)$ vanishes in $(i,h)$ if and only if both $(k_{1i},k_{2i})$ and $(v_{1i},v_{2i})$ vanish in $(i,h)$ and;
			\item if $(v_1,v_2)$ is $m_H$-globally unsolvable in $(i,h)$ then $(k_{1i},k_{2i})$ vanishes in $(i,h)$ but $(v_{1i},v_{2i})$ does not.
		\end{itemize}
		
		To find these elements, fix $i\in \range{l}$. Define $\mc H_i\subset C(g)$ by the condition $h\in \mc H_i$ if $v_1,v_2$ vanishes in $(i,h)$ or is $m_H$-globally unsolvable in $(i,h)$, which is hence equivalent to $$\pi_i(v_1-\rho(h)v_2)\in V_g(M_i)+\abs{G}^{m_H}M_i.$$

		If $\mc H_i$ is empty, then we define $v_{1i}=v_{2i} = 0$. In the other case, fix an element $h_i \in \mc H_i$. 
		We demonstrate that $\mc H_ih_i\inv$ is characterized as $h\in\mc H_ih_i\inv$ if and only if  $$\pi_i(\rho(h_i)v_2-\rho(h)\rho(h_i)v_2)\in V_g(M_i)+\abs{G}^{m_H}M_i.$$ Indeed, $$\pi_i(\rho(h_i)v_2-\rho(h)\rho(h_i)v_2)=\pi_i(v_1-\rho(hh_i)v_2)-\pi_i(v_1-\rho(h_i)v_2)$$
		The last term belongs to $V_g(M_i)+\abs{G}^{m_H}M_i$, leading to the equivalence.
		
		We may now apply Lemma \ref{prop:representationsAreScalable:lemma} on $\pi_i(\rho(h_i)v_2)+V_g(M_i)$ and the subgroup $\mc H_ih_i\inv$. Doing so gives us some $k, v \in M^\prime$ such that $$
		k+\abs{G}^5v+V_g(M_i)=\pi_i(\rho(h_i)v_2)+V_g(M_i)
		$$
		and such that $\rho(h)k=k$ for all $h\in\mc H_ih_i\inv$. Notice that by Lemma \ref{prop:modOutVg:lemma} applied to the group $C(g)$ $$\varpi_{V_g(M_i)}(\abs{G} v)=\varpi_{V_g(M_i)}\left(\frac{\abs{G}}{\abs{C(g)}}\varpi_{V_g(M_i)}(v)\right),$$ and thus $\abs{G} v$ and $\frac{\abs{G}}{\abs{C(g)}}\varpi_{V_g(M_i)}(v)$ differ only by an element of $V_g(M_i)$. We can thus further rewrite $\pi_i(\rho(h_i)v_2)+V_g(M_i)$ as $k+\abs{G}^4\tilde v+V_g(M_i)$ where $k$ is as before and $\rho(g)(\tilde{v}) = \tilde{v}$. By Lemma \ref{prop:squareHullAlmostIncluded:lemma}, we have that $v^\prime = \abs{G}\tilde{v} \in M \cap M_i$. We now define $v_{2i}=\rho(h_i\inv)v'$ and $k_{2i}=v_2-\abs{G}^3v_{2i}$.
		
		As $(v_1,v_2)$ is either $g,m_H$ globally unsolvable, or $g$-vanishes in $(i,h_i)$, it follows that $\pi_i(v_1-\rho(h_i)v_2)\in V_g(M_i)+\abs{G}^{m_H}M$. We may thus find $v_{0i}\in M\cap M_i$ with $\rho(g)(v_{0i}) = v_{0i}$ and such that $\pi_i((v_1-\abs{G}^3v_{0i})-\rho(h_i)v_2)\in V_g(M_i)$. Define ${v_{1i}=v_{0,i}+v'}$ and define $k_{1i}=v_1-\abs{G}^3v_{1i}$. Notice that by construction, $\pi_i(k_{1i}-\rho(h_i)k_{2i})\in V_g(M_i)$.
		
		
Having these elements for all indices $i$, we now define \begin{align*}
			v_1'&=\sum_{i=1}^l v_{1i} \in M\\
			v_2'&=\sum_{i=1}^l v_{2i} \in M\\
			k_1=&v_1-\abs{G}^3v_1' = \sum_{i=1}^l k_{1i}\\
			k_2=&v_2-\abs{G}^3v_2'= \sum_{i=1}^l k_{2i}.
		\end{align*}
		Note that $(k_1,g), (k_2,g) \in H$, as $\abs{G}M \subset H$. We will demonstrate that these elements have the properties mentioned in the beginning of the proof.
		
		First assume that $(v_1,v_2)$ either vanishes or is $m_H$ globally unsolvable in $(i,h)$.
		This implies that $h\in\mc H_i$. By construction of $k_1$ and $k_2$, we have that $\pi_i(k_1)=k_{1i}$ and $\pi_i(k_2) = k_{2i}$. Furthermore $k_{2i}$ is such that ${\pi_i(\rho(h)k_2-\rho(h_i)k_2)\in V_g(M_i)}$. Also using that ${\pi_i(k_1-\rho(h_i)k_2)\in V_g(M_i)}$, it follows that $(k_1,k_2)$ vanishes in $(i,h)$.
		In the case that $(v_1,v_2)$ vanishes in $(i,h)$, it additionally follows that $(\abs{G}^3v_1',\abs{G}^3v_2')$ and thus also $(v_1',v_2')$ vanishes in $(i,h)$. Conversely, if $(k_1,k_2)$ and $(v_1',v_2')$ both vanish in $(i,h)$, then so does $(v_1,v_2)$, leading to the first property. 
		
		Notice furthermore that $k_1$ and $v_1$ differ by an element of $\abs{G}^3M$. The same holds for $k_2$ and $v_2$. It follows that $(v_1,v_2)$ is strongly unsolvable in $h$ if and only if $(k_1,k_2)$ is.	
		
		We demonstrate that $\range{l}$ $g$-admits $(v_1',v_2',k_1,k_2)$. If this would not the case, then there exists some $h\in C(g)$ such that $(k_1,k_2)$ is not strongly unsolvable, and such that for all $i\in\range{l}$, both $(v_1',v_2')$ and $(k_1,k_2)$ vanish in $(i,h)$. This would however imply that $(v_1,v_2)$ is not strongly unsolvable in $h$, and that for all $i$, the pair $(v_1,v_2)$ vanishes in $(i,h)$ which implies that $\range{l}$ does not $g,m_H$-admit $(v_1,v_2)$. We conclude via contradiction that indeed $\range{l}$ does $g$-admit $(v_1',v_2',k_1,k_2)$.
		
		Now let $K\subset\range{l}$ be such that $K$ $g$-admits $(v_1',v_2',k_1,k_2)$.
		For $h\in C(g)$ one of the following must be true.\begin{itemize}
			\item The pair $(k_1,k_2)$ is strongly unsolvable in $h$, which implies that $(v_1,v_2)$ is also strongly unsolvable in $h$.
			\item There exists some $i\in\range{l}$ such that $(k_1,k_2)$ is weakly unsolvable in $(i,h)$. In this case $(v_1,v_2)$ cannot vanish, or be $m_H$-globally unsolvable in $(i,h)$ thus the pair $(v_1,\label{key}v_2)$ is $m_H$-locally unsolvable in $(i,h)$.
			\item There exists some $i\in K$ such that $(v_1',v_2')$ is weakly unsolvable in $(i,h)$. In this case $(v_1,v_2)$ may not vanish in $(i,h)$ and thus must $(v_1,v_2)$ be either $m_H$-globally or $m_H$-locally unsolvable.
		\end{itemize}
		In each of the above cases it follows that $K$ must $g,m_H$-admit $(v_1,v_2)$, concluding the proof.
		
		\end{proof}
		\section{Lower bound}\label{sec:3->1}
		In this section we will demonstrate the following:
		\begin{proposition}
			Let $\rho:G\acts M$ be a representation of a finite group on a free $\ZZ$-module. Let $H<M\rtimes G$ be an extension of $\abs{G}M$ by $G$.
			Then $$\conj_H(n)\succ \ln^{\mf k_3}(n)$$
			where $\mf k_3$ is given by
			$$\mf k_3=\max_{g\in G}\max_{\substack{(k_1,g),(k_2,g)\in H\\v_1,v_2\in M}}\min\set{\dim K\mid K \text{ }g\mathrm{-admits} (v_1,v_2,k_1,k_2)}.$$
		\end{proposition}
		 
		In fact, the proof works by constructing an explicit sequence of elements that are non-conjugate but whose conjugacy classes can only be separated in sufficiently large quotients. In order to find these elements, we can assume that $\mf k_3>0$ as $\conj_H$ is always at least constant. From the definition of $\mf k_3$ we can fix an element $g \in G$ and $(v_1,v_2,k_1,k_2)$ such that $\mf k_3=\min\set{\dim K\mid K \text{ }g\mathrm{-admits} (v_1,v_2,k_1,k_2)}$. As the statement of the proposition is invariant under changing the generating set of $H$, we can choose $S$ such that it contains the elements $$\set{(k_1+\abs{G}^4v_1,g),(\abs{G}^4v_1,1),(k_2+\abs{G}^4v_2,g),(\abs{G}^4v_2,1)}.$$
		 
		The sequence of elements is given by \begin{align*}
		 (v_1(n),&v_2(n))_{n\in\NN}=\\
		 &\left((k_1+\abs{G}^4\kgv\range nv_1,g),(k_2+\abs{G}^4\kgv\range nv_2,g)\right)_{n\in\NN}.
		 \end{align*}
		 Notice that $(v_1(n),g)$ has word norm at most $\kgv\range{n}$ with respect to $S$ as $(v_1(n),g)$ can be written as$$
		 (\abs{G}^4v_1,1)^{\kgv\range{n}-1}(k_1+\abs{G}^4v_1,g).
		 $$
		 Similarly, $(v_2(n),g)$ has word norm at most $\kgv\range n$. These elements might be conjugate, but only for finitely many $n \in \NN$.
		 
		 \begin{lemma}\label{prop:conjfinitelyoften:lemma}
		 	The set $\mf N = \left\{ n \in \NN \mid v_1(n) \sim v_2(n) \right\}$ is finite.
		 \end{lemma}
		 \begin{proof}
		Assume for a contradiction that $\mf N$ is infinite. If $v_1(n)$ and $v_2(n)$ are conjugate, then it follows from Lemma \ref{prop:conjugatesinrepresentation:lemma} that there exists some $h'\in C(g)$ such that 
		 \begin{equation}v_1(n)-\rho(h')v_2(n) - (v_{h'}-\rho(g)v_{h'})\in \abs{G}W_g(M)\label{eq:conjfinitelyoften:conjugate}\end{equation}
		 Let $\mf N_h$ be the subset of $\mf N$ such that the above holds for $h=h'$. By the pigeonhole principle, there exists $h \in C(g)$ such that $\mf N_h$ is infinite.
		 
		 For $n_1,n_2\in\mf N_h$, by subtracting the above equations, we thus obtain $$
		 v_1(n_2)-v_1(n_1)-\rho(h)(v_2(n_2)-v_2(n_1)))\in \abs{G}W_g(M)
		 $$
		 from which it follows that \begin{equation}\abs{G}^4 (\kgv\range{n_2}-\kgv\range{n_1})(v_1-\rho(h)v_2)\in \abs{G}W_g(M).\label{eq:conjfinitelyoften:twiceconjugate}\end{equation} In particular it follows that for any $i\in\range{l}$, the pair $(v_1,v_2)$ vanishes in $(i,h)$.
		 Take now $n_3\in\mf N_h$ sufficiently large such that ${\kgv\range{n_2}-\kgv\range{n_1}}$ divides $\kgv\range{n_3}$.
		 Subtracting a multiple of equation (\ref{eq:conjfinitelyoften:twiceconjugate}) from equation (\ref{eq:conjfinitelyoften:conjugate}) for $n=n_3$, it follows that $$
		 k_1-\rho(h)k_2-(v_h-\rho(g)v_h)\in \abs{G}W_g(M)		 
		 $$
		 From this it follows both that $(k_1,k_2)$ is not strongly unsolvable in $h$ and that for all $i\in \range l$, $(k_1,k_2)$ vanishes in $(i,h)$.
		 All this demonstrates that $\range{l}$ does not $g$-admit $(v_1,v_2,k_1,k_2)$ which is in contradiction with the assumption that $\mf k_3\geq 1$.
		 \end{proof}
		 
		The remainder of this section is showing that the quotients needed to separate the conjugacy classes of the elements $v_1(n)$ and $v_2(n)$ for $n \notin \mf N$ are sufficiently large. For every $n\notin\mf N$, let $\psi_n:H\rightarrow Q$ be a group quotient to a finite group $Q$, such that $\psi_n(v_1(n))$ and $\psi_n(v_2(n))$ are non conjugate. As $H$ is conjugacy separable such a morphism exists. Let $N'$ be the kernel of this map and define $N=N'\cap M$, as $N'$ and $M$ are both normal, it follows that $N$ is as well. In particular $N$ must be closed under conjugation with $(v_h,h)$, similarly as shown in \cite[Theorem 3.3.]{dere2023residual}, which implies that $N$, seen as a submodule of $M$, is $\rho$-invariant.
		
	The strategy is now to associate to $N$ some subset $K\subset\range{l}$ such that $K$ $g$-admits $(v_1,v_2,k_1,k_2)$ and such that $N$ has index at least $Cn^{\dim K}$ where $C$ is some constant only depending on $G\acts M$ but \textbf{not} on $n$.	As $\psi_n(v_1(n))$ and $\psi_n(v_2(n))$ are non-conjugate, it certainly must be the case that $v_1(n)N$ and $v_2(n)N$ must be non conjugate, or thus thanks to Corollary \ref{prop:conjugatesinrepresentationquotient:corollary} it must hold for all $h$ that $$
		 v_2(n)-\rho(h)v_1(n)-(v_h-\rho(g)v_h)\notin \abs{G}W_g(M)+N
		 $$
		 Define $\mf H \subset C(g)$ as the subset consisting of those $h \in C(g)$ such that for all $i\in\range{l}$, the pair $(k_1,k_2)$ is not strongly unsolvable in $h$, and vanishes in $(i,h)$.
		 This implies that $$k_1-\rho(h)k_2-(v_h-\rho(g)v_h)\in V_g(M)\cap (\abs{G}W_g(M)+\abs{G}^3M)=\abs{G} W_g(M)$$ for all $h \in \mf H$. In particular, we thus have that $$
		 \abs{G}^4\kgv\range{n}(v_1-\rho(h)v_2)\notin N+\abs{G}W_g(M).
		 $$
		 
		From now on we write $\varpi_i$ for the maps from Remark \ref{prop:repindicatorZstrong:remark} applied to the irreducible subrepresentation corresponding to the $i$th component. As $N$ is $\rho$-invariant, it follows that $N\supset \varpi_1(N)\oplus\cdots\oplus \varpi_l(N)$ and by Remark \ref{prop:repindicatorZstrong:remark} also that $\varpi_1(N)\oplus\cdots\oplus \varpi_l(N)\supset\abs{G}N$.
		 We thus have \begin{align*}\abs{G}^4\kgv\range{n}&(v_1-\rho(h)v_2)\notin\\ &\left((\varpi_1(N)+ \abs{G}^2W_g(M_1))\oplus\cdots\oplus(\varpi_l(N)+ \abs{G}^2W_g(M_l))\right)\end{align*}
		 In particular there must exist some $i\in\range{l}$ such that $$\pi_i(\abs{G}^4\kgv\range{n}(v_1-\rho(h)v_2))\notin \varpi_i(N)+ \abs{G}^2W_g(M_i)$$ 
		 Choose for every $h \in \mf H$ one such index $i_h$ such that the previous holds.

		 Let $K=\set{i_h\mid h\in \mf H}$. By construction of $i_h$, it follows that $$\pi_{i_h}(\abs{G}^4\kgv\range{n}(v_1-\rho(h)v_2))\notin \abs{G}^2W_g(M_{i_h})$$ and by Lemma \ref{prop:WRadicalInSubgroup:lemma} it follows that $$\pi_{i_h}(\abs{G}^4\kgv\range{n}(v_1-\rho(h)v_2))\notin V_g(M_{i_h})$$ or thus that $(v_1,v_2)$ is weakly unsolvable in $i_h$.
		 In particular it follows that $K$ $g$-admits $(v_1,v_2,k_1,k_2)$ and thus that $\dim(K)\geq \mf k_3$.

		 We also have that $$\pi_i(\abs{G}^4\kgv\range{n}(v_1-\rho(h)v_2))\notin \varpi_i(N)\subset\varpi_i(N)+ \abs{G}^2W_g(M_i).$$ 
		 
		 \begin{lemma}
		 	The index of $\varpi_i(N)$ in $M_i$ is at least $C_in^{d_i}$ where $C_i$ is some non-zero constant not depending on $n$ and $d_i$ is the dimension of one of the irreducible subrepresentations of $M_i\otimes\CC$.
		 \end{lemma}
		 \begin{proof}
		 	This follows from the methods of the lower bound of \cite[Theorem 1.2]{dere2023residual}. This theorem relies on \cite[Theorem 5.15]{dere2023residual} which is proven for $\gcd\range{n}v_0$ where $v_0$ is the first basis vector. Nothing in the further proof however prevents us from swapping $v$ with any vector or thus in our case $\abs{G}^4\kgv\range{n}(v_1-\rho(h)v_2)$.
		 \end{proof}

		 The index of $\bigoplus_{i\in K} \varpi_i(N)$ is thus at least $Cn^{\sum_{i \in K} d_{i}}$ where  $C$ is a constant not depending on $n$. It also follows that $N$ has an index at least $\frac{C}{\abs{G}^{(\dim M)^2}}n^{\sum_{i \in K} d_{i_h}}$ in $\abs{G}M$ and thus that $N'$ has index at least $\frac{C}{\abs{G}^{1+(\dim M)^2}}n^{\sum d_{i_h}}$ in $H$.
		 As $\sum_{i \in K} d_{i_h}$ is precisely $\dim K$, it thus follows that $\conj_H(v_1(n),v_2(n))\geq C'n^{\mf k_3}$ for some $C^\prime > 0$.
		 Remember that $(v_1(n),g)$ and $(v_2(n),g)$ have a word norm of at most $\kgv\range n$ with respect to $S$. From this it thus follows that $$\conj_H(\kgv\range n)\succ n^{\mf k_3}\succ \ln^{\mf k_3}(\kgv\range n).$$ The proposition now follows from Lemma \ref{prop:logarithmicLowerBound:lemma}.

		 	
		 	\section{Upper Bound}
		 	\label{sec:1->2}
		 	In this section, we complete the proof of Theorem \ref{prop:conjugacyseparabilittysplitvirtuallyabelian:theorem} by proving that $\conj_H(n)\prec \ln^{\mf k_2}(n)$, with notations as in the result itself.
		 	
		 	A crucial part of the proof will be the construction of a polynomial which will allow us to find finite quotients to separate conjugacy classes. The following notation and observation will then be important.
		 	\begin{definition}\label{def:upperboundpoly}
		 		Let $f\in\ZZ[x_1\cdots x_n]$ be a polynomial, written as $$f(x_1,\cdots x_n)=\sum_{i_1}\cdots\sum_{i_n} a_{i_1,\cdots,i_n}x_1^{i_1}x_2^{i_2}\cdots x_n^{i_n}$$ with coefficients $a_{i_1,i_2,\cdots,i_n} \in \ZZ$. We define the polynomial $\vert f \vert \in\ZZ[x]$ as $$
		 		\vert f \vert (x)=\sum_{i_1}\cdots\sum_{i_n} \abs{a_{i_1,\cdots,i_n}}x^{i_1+i_2+\cdots+i_n}
		 		.$$
		 	\end{definition}
		 	\noindent The following estimate for a polynomial follows from a direct computation.
		 	\begin{lemma}
		 		\label{prop:monovariablebound}
		 		Let $f\in\ZZ[x_1\cdots x_m]$ be a  polynomial. For any $n \in \NN_0$ and integers $a_i\in[-n,n]$ it holds that $f(a_1,\cdots,a_m)\leq \vert {f} \vert (n)$
		 	\end{lemma}
		 	
		 In the next lemma, we construct surjective ring morphisms that we will need in the proof.
		
		 	\begin{lemma}\label{prop:modpexists}
		 		Let $n\in\NN_0$ and let $p$ be prime such that $p\cong 1\mod n$, then there exists a surjective morphism of rings $\ZZ[\zeta_n]\rightarrow \ZZp$.
		 	\end{lemma}
		 	\begin{proof}
				Recall that $\ZZ[\zeta_n]$ is as a ring isomorphic with $\frac{\ZZ[x]}{(\Phi_n(x))}$, where $\Phi_n$ is the $n$th cyclotomic polynomial. 
				
				We first construct a morphism from $\ZZ[x]$ to $\ZZp$. The group of units $\left(\ZZp\right)^\times$ is a cyclic group of order $p-1$. As $n$ divides $p-1$, it follows that there exists some unit $u\in\left(\ZZp\right)^\times$ of order $n$. Let $\varphi:\ZZ[x]\rightarrow \ZZp$ be the unique ring morphism such that $\varphi(1)=1$ and $\varphi(x)=u$. 
				
				To show that $\varphi$ induces a morphism $\varphi:\ZZ[\zeta_n]\rightarrow \ZZp$, we need to show that $\varphi$ maps the ideal $\left(\Phi_n(x)\right)$ to $0$. It suffices thus to show that $\Phi_n(x)$ is mapped to $0$. As $u$ is of order $n$, it is clear that $\varphi(x^n-1)=0$. The polynomial $x^n-1$ factors as $\prod_{i\mid n}\Phi_i(x)$ over the integers. If we can demonstrate that $\varphi(\Phi_i(x))\neq 0$, whenever $i\mid n$ but $i\neq n$, then we are done as $\ZZp$ is a field without zero divisors. Notice that $\Phi_i(x)\mid x^{i}-1$ and thus $\varphi(\Phi_i(x))\mid u^{i}-1$. As $u$ is of order $n$ and $i$ is strictly less then $n$, it follows that $u^{i}-1$ is non-zero, and thus also $\varphi(\Phi_i(x))$ non-zero.
		 	\end{proof}

		 	The maps $\varpi$, introduced in Section \ref{sec:prereq}, allow us to pass from representations to certain subrepresentations. However, this only works if these correspond to different characters. If this is not the case then we can use the following lemma, which we will only use for either $\ZZ$ or $\frac{\ZZ}{p\ZZ}$.
		 	\begin{lemma}\label{prop:identicalcoppies:lemma}
		 		Let $m$ be a positive integer and $R$ an integral domain of order at least $m+1$. Take $M_1$ any $G$-representation on $R$ with $V_1$ a radical subspace of $M_1$. For any $l \geq 0$ we denote by $M_l=M_1^l=M_1\oplus\cdots\oplus M_1$ and $V_l=V_1^l\subset M_l$. Then there exists a finite set of $G$-maps $\left\{\varphi_i:M_l\rightarrow M_1\right\}_{i \in I}$, such that for every subset $\mc M \subset M_l$ of size at most $m$ with $\mc M\cap V_l=\emptyset$, there exists some $i \in I$ such that $\varphi_i(\mc M)\cap V_1=\emptyset$.
		 	\end{lemma}
		 
		 	\begin{proof}
		 		When $l=1$, the result holds for the singleton containing the identity map. By induction it thus suffices to find maps $\varphi_i: M_{l+1} \to M_{l}$ satisfying the conditions of the lemma, by composing these with the maps $M_l \to M_1$ from the induction hypothesis. So the remainder of the proof will be constructing maps $M_{l+1} \to M_l$ satisfying the property of the statement. In fact, we will show that any $m+1$ different maps of a certain type will satisfy this property.
		 		
		 		For any $\lambda\in R$ we define the map $\varphi_\lambda:M_{l+1}\rightarrow M_{l}$ by $$\varphi_\lambda(v_0,v_1,v_2,\cdots,v_l)= (v_1-v_0,v_2-v_0,\cdots,v_l-\lambda v_0).$$
		 		Notice that this map is a $G$-map. Furthermore, this map is surjective and its kernel is given by$$
		 		K_\lambda=\set{(v,v,\cdots,\lambda v)\in M_{l+1}\mid v\in M_1}.
		 		$$
		 		A computation shows that $\varphi_\lambda\inv(V_l)=K_\lambda+V_{l+1}$.
		 		
		 		We will demonstrate that if $\lambda_1$ and $\lambda_2$ are distinct, then $$
		 		\left(K_{\lambda_1}+V_{l+1}\right)\cap\left( K_{\lambda_2}+V_{l+1}\right)=V_{l+1}.
		 		$$ 	 		
		 		Take $k_1=(v_1,v_1,\cdots  \lambda_1v_1) \in K_{\lambda_1}$ and $k_2=(v_2,v_2,\cdots,\lambda_2v_2) \in K_{\lambda_2}$ such that $k_1+V_{l+1}=k_2+V_{l+1}$, or in other words such that $k_1-k_2\in V_{l+1}.$
		 		Projecting onto the first coordinates, implies that $v_1-v_2\in V_1$. Similarly, projecting onto the last coordinate implies that $\lambda_1v_1-\lambda_2v_2\in V_1$. Combining these two, it follows that $(\lambda_2-\lambda_1)v_2\in V_1$. As $\lambda_1$ and $\lambda_2$ are distinct, and as $V_1$ is radical, it follows that $v_2\in V_1$ (and thus also $v_1\in V_1$). By definition of $V_{l+1}$, it follows that $k_1$ and $k_2$ both lie in $V_{l+1}$, leading to the claim above.
		 		
		 		Now fix any subset $\Lambda \subset R$ of order $m+1$, then we show that the maps $\varphi_\lambda$ with $\lambda \in \Lambda$ satisfy the property of the lemma. Let $\mc M$ be an arbitrary subset of $M_{l+1}$ of order at most $m$ with $\mc M \cap V_{l+1} = \emptyset$. By the pigeonhole principle, there exists some $\lambda\in\Lambda$ such that none of the elements of $\mc M$ ly in $K_\lambda+V_{l+1}$, and thus such that $\varphi_\lambda(\mc M)\cap V_l=\emptyset$. In other words, the maps $\varphi_\lambda$ with $\lambda \in \Lambda$ satisfy the property of the lemma.
		 	\end{proof}
		 	
		 	The following lemma talks about separating a finite set from $V_g(M)$ in a finite quotient whenever $M$ is irreducible. The main example of $\mf A$ in what follows are subsets of the set $\set{v_1-\rho(h)v_2\mid h\in C(g)}$. In what follows, we will assume that a finite generating set on $M$ is given, meaning that the norm of elements $m \in M$ is defined. For a finite set $\mf A$, we define $\Vert \mf A \Vert = \max \set{ \norm{a} \mid a \in \mf A}$.

		 	\begin{lemma}\label{prop:conjsepnotinW':lemma}
		 		Let $\rho:G\acts M$ be a $\ZZ$-irreducible representation of a finite group on a free $\ZZ$-module. Let $d$ be the dimension of the $\CC$-irreducible subrepresentations of $M\otimes \CC$. For every $g\in G$ there exists some constant $C$ such that for any $\mf A\subset M$ with $\mf A\cap V_g(M)=\emptyset$ and $\abs{\mf A}\leq \abs{G}$ there exists a finite action-preserving quotient $\pi:M\rightarrow Q$ such that $\vert Q \vert \leq C + C \ln^d(\Vert \mf A \Vert)$ and such that $\pi(a)\notin\pi(V_g(M))$ whenever $a\in \mf A$.
		 		
		 	\end{lemma}
		 	\begin{proof}
		 		Fix an element $g \in G$ and let $\iota$ be the embedding of $C(g)$ into $G$ with corresponding $C(g)$-representation $\kappa=\rho\comp\iota$. The representation 	 		$\kappa$ is not necessarily irreducible and thus we can write $M\otimes\QQ=\displaystyle \bigoplus_{j\in J}K_j$ for some index set $J$, where $K_j$ is the subspace by combining all isomorphic $\QQ$-irreducible subrepresentations of $\kappa$ as before. In particular, if $i \neq j$, then $K_i$ and $K_j$ do not have isomorphic subrepresentations. Denote now $B=\set{j\in J\mid \kappa_j(g)=\id_M}$, that is, the set of indices $j$ such that $g$ acts trivially under the representation $\kappa_j$. By Lemma \ref{prop:structOfVgM:lemma}, $V_{g}(M)$ can be rewritten as $$M\cap \sqrt{\bigoplus_{j\notin B} K_j}.$$

		 		Consider the map $\varpi_{V_{g}(M)}=\displaystyle \prod_{j\notin B}(c_{\kappa_j}\id_M-\varpi_{\kappa_j})$ where the product means composition of maps. Notice that this is well-defined, as Lemma \ref{prop:indicommut} guarantees that the order of composing the maps does not matter. This map has $V_{g}(M)$ as kernel, furthermore, on $\displaystyle \bigoplus_{j\in B}K_j$,  this map acts as multiplication with the constant $\displaystyle \prod_{j\notin B}c_{\rho_j}$.
		 		Up to a constant, this is thus a projection on $\displaystyle \bigoplus_{j\in B}K_j$ with $V_{g}(M)$ as kernel.
		 		
		 		Let the decomposition into irreducible representations of $\rho\otimes\QQ[\zeta_G]$ be $$
		 		\rho\otimes\QQ[\zeta_G]=\bigoplus_{i\in I}\rho_i^m
		 		$$
		 		and let $M\otimes\QQ[\zeta_G]=\displaystyle \bigoplus_{\substack{i\in I\\1\leq l\leq m}}N_{i,l}$ be such that $N_{i,l}$ are irreducible subspaces on which $\rho$ acts like $\rho_i$ and let $N_i=\bigoplus_{1\leq l\leq m}N_{i,l}.$
		 		For $i\in I$, we now define the maps $f_i:M\otimes\ZZG\rightarrow M\otimes\ZZG$ as follows:$$
		 		f_i=\varpi_{\rho_i}\comp\varpi_{V_{g}(M)}
		 		$$
		 		Note that thanks to Lemma $ \ref{prop:indicommut}$,  the map $\varpi_{V_{g}(M)}$and the map $\varpi_{\rho_i}$ commute, or thus that $f_i=\varpi_{V_{g}(M)}\comp\varpi_{\rho_i}$.
		 		The maps $f_i$ are linear maps $M\otimes\ZZG\rightarrow M\otimes\ZZG$ as composition of maps of the form $\varpi_\rho'$ and $(c\id_{\rho'}-\pi_{\rho'})$, all of which are linear. 
		 		
		 		Fix $\beta$ a basis of $M$. For any $\sigma\in\mathrm{Gal}(\QQ[\zeta_G]/\QQ)$, we consider $\sigma(f_i)$, which is just applying the transformation $\sigma$ to all entries in the basis $\beta$. This can be decomposed as $\sigma(\varpi_{\rho_i})\comp \sigma(\varpi_{V_{g}(M)})$. The map $\varpi_{V_{g}(M)}$ restricts to a linear map from $M$ to $M$ and is as such invariant under $\sigma$. The map $\sigma(\varpi_{\rho_i})$ however, thanks to Theorem \ref{thm:shur} is of the form $\varpi_{\rho_{i'}}$, where $i$ and $i'$ may be distinct. It thus follows that $\sigma(f_i)=f_{i'}$ for some $i'\in I$.
		 		
		 		Consider then the function $F$ obtained by coordinate multiplication of the functions $f_i$. That is if for $b \in \beta$ the maps  $f_{ib}:M\otimes\ZZG\rightarrow \ZZG$ are such that $f_i=\sum_{b\in\beta}f_{ib}b$, then $F$ is the function given by $\sum_{b\in\beta}\left( \prod_{i\in I} f_{ib}\right)b$, where the product is taken point-wise this time. As $f_i$ all are linear with entries in $\ZZG$, it follows that $F: M \otimes \ZZG \to M \otimes \ZZG$ is polynomial with coefficients in $\ZZG$. We once again apply a Galois transformation $\sigma\in\mathrm{Gal}(\QQ[\zeta_G]/\QQ)$ to $F$. As $\sigma$ just permutes the indices of $i$, it follows that $\sigma(F)=F$. By the fundamental theorem of Galois theory, we thus have that $F$ has integer coefficients and thus can be considered as a polynomial map $F: M \to M$. 
		 		
		 		We will first show that the kernel of $F$ is exactly $V_{g}(M)$.
		 		As $V_{g}(M)$ is the kernel of $\varpi_{V_{g}(M)}$, it is clear that $V_{g}(M)$ lies in the kernel of $f_i$ for every $i \in I$. It thus also follows that $V_{g}(M)$ must lie in the kernel of $F$. Conversely, consider $v\in M$ and suppose $F(v)=0$. Then all the functions $f_i(v)$ must be $0$. Indeed suppose that $f_i(v)$ is non-zero, then there exists some $b\in \beta$ such that $f_{ib}(v)\neq0$. As $\beta$ is invariant under Galois actions, it would follow that for any $\sigma\in\mathrm{Gal}(\QQ[\zeta_G]/\QQ)$ that $\sigma(f_{ib})(v)\neq0$ and thus that $\prod_{i\in I}f_{ib}(v)$ is non-zero. This contradicts the fact that $F(v)$ must be $0$.
		 		Thus $f_i(v)=0$ for all $i\in I$. Consider now the map $\sum f_i$.
		 		Notice that this is $(\sum_{i\in I}\varpi_{\rho_i})\comp\varpi_{V_{g}(M)}$. Furthermore notice that the map $\sum_{i\in I}(\varpi_{\rho_i})$ is just multiplication by some non-zero constant as $\rho$ is irreducible over $\QQ$. As $f_i(v)=0$, it follows that $\sum_{i\in I} f_i(v)$ must be $0$ and thus must $\varpi_{V_{g}(M)}(v)=0$, or thus that $v\in V_{g}(M)$.
		 		
		 		Let $p$ be a prime number congruent to $1\mod\abs{G}$, we will study the image of $v_0$ in $\frac{M}{pM}=M\otimes\ZZp$.
		 		Let $\varphi:\ZZG\rightarrow \ZZp$ be the ring-epimorphism of Lemma \ref{prop:modpexists}. We then have the following diagram of ring morphisms.
		 		\[\begin{tikzcd}[ampersand replacement=\&]
		 			\ZZ \& \ZZG \\
		 			\ZZp
		 			\arrow["1"', from=1-1, to=2-1]
		 			\arrow["1", from=1-1, to=1-2]
		 			\arrow["\varphi", from=1-2, to=2-1]
		 		\end{tikzcd}\]
		 		The horizontal and vertical arrow are the unique ring morphism from $\ZZ$ to the other ring. By unicity of the vertical arrow, it follows that the diagram commutes.
		 		This induces the following commutative diagram of $G$-modules.
		 		\[\begin{tikzcd}[ampersand replacement=\&]
		 			M\otimes\ZZ \& M\otimes\ZZG \\
		 			M\otimes\ZZp
		 			\arrow["\id_M\otimes1"', from=1-1, to=2-1]
		 			\arrow["\id_M\otimes1", from=1-1, to=1-2]
		 			\arrow["\id_M\otimes\varphi", from=1-2, to=2-1]
		 		\end{tikzcd}\]
		 		The horizontal and vertical arrows are a morphism of $\ZZ[G]$-modules. The diagonal map is a $G$-map of $\ZZG$-modules. The maps $\varpi_{\rho_i}$ and $\varpi_{\kappa_j}$ are combinations of the group action and scalar multiplication, thus they commute with the morphisms in this diagram. 
		 		
		 		The map $\bar{\varphi} = \id_M \otimes \varphi$ cannot be applied immediately on any of the irreducible components $N_{i,l}$, however we can do this on the latices $N_{i,l}\cap \left( M\otimes\ZZG\right)$. For $\rho$-invariant lattices $N$ and $N'$ in $N_{i,l}$, the induced irreducible representations $\bar{\varphi}(N)$ and $\bar{\varphi}(N')$ over $\ZZp$ have the same characters, hence it follows by \cite[Corollary 9.22]{isaacs2006character} that they are isomorphic over $\ZZp$. In particular, we have that $\bar{\varphi}\left(N_{i,l}\cap \left(M\otimes\ZZG\right)\right)$ and $\bar{\varphi}\left(N_{i,l'}\cap \left(M\otimes\ZZG\right) \right)$ are always isomorphic.  		In what follows we just denote $\bar{\varphi}(N_{i,l})$ or $\bar{\varphi}(N_{i})$ instead of the more cumbersome $\bar{\varphi}\left(N_{i,l'}\cap \left(M\otimes\ZZG\right) \right)$ or $\bar{\varphi}\left(N_{i,l'}\cap \left(M\otimes\ZZG\right) \right)$. Combining \cite[Corollary 9.15]{isaacs2006character} with \cite[Lemma 15.5]{isaacs2006character}, we obtain that $\bar{\varphi}(N_i)$ is isomorphic to $\bar{\varphi}(N_{i,1})^m$

		 		On $M\otimes\ZZG$, the maps $\varpi_\rho$ behave like a projection times a constant. We demonstrate that the induced maps on $M \otimes \ZZp$, which we denote as $\bar{\varphi}(\varpi_\rho)$, behave similarly. For this, we study the behavior of $\bar{\varphi}(\varpi_{\rho_i})$ on $\bar{\varphi}(N_{j,l})$ for $i,j\in I$. First if $i \neq j$, then we have $\varpi_{\rho_i}(N_{j,l})=0$. It thus follows that $\bar{\varphi}(\varpi_{\rho_i})(\bar{\varphi}(N_{j,l}))=0$. Now suppose that $i=j$. Then $\varpi_{\rho_i}\mid_{N_{j,l}}$ is multiplication with some constant $c_{\rho_i}$. The number $c_{\rho_i}$ is a divisor of $\abs{G}$ and thus coprime with $p$. It follows that $\bar{\varphi}(\varpi_{\rho_i})\mid_{\bar{\varphi}(N_{j,l})}$ is multiplication with some unit.
		 		Similar results hold for the maps $\varpi_{\kappa_j}$. Using that $\bar{\varphi}$ and $\varpi_\chi$ always commute, it follows that $\bar{\varphi}(M)$ decomposes as $\bigoplus_{i\in I} \bar{\varphi}(N_i)$. where the projection maps are given (up to a constant) by $\varpi_{\rho_i}$
		 		

		 		Let $v\in M\otimes \ZZG$ be such that for some $i$ holds that $\bar{\varphi}\comp\varpi_{\rho_i}(v)\in V_{g}(M)\otimes\ZZp=\bar{\varphi}(V_{g}(M))$. Then $\bar{\varphi}\comp\varpi_{V_{g}(M)}\comp\varpi_{\rho_i}(v)=0$ and thus $\bar{\varphi}(f_i(v))=0$. As $M\otimes\ZZG$ is a free $\ZZG$-module over $\beta$, this would further imply that all images $f_{ib}(v)$ belong to the kernel of $\bar{\varphi}$ and thus that $\bar{\varphi}(F(v))=0$. 		
		 		If furthermore $v\in M$, then $F(v)$ also belongs to $M$ and this thus implies that $F(v)\in pM$.
		 		
		 		Let $p$ now be a prime such that for any $a\in \mf A$, the prime $p$ is not a divisor of the index of $F(a) \in M$. By Lemma \ref{prop:monovariablebound} we have that this index is no larger then $\vert F \vert (n)$. By Lemma \ref{prop:pnondevider}, such a $p$ can be chosen no larger then $C_{\abs G}+C_{\abs G}\ln( (\vert F \vert (n))^{\abs{\mf A}})$ for some uniform constant $C_{\abs G}.$ Then by the previous, $\varpi_{\rho_i}(a)+pM$ all lie in the space $\bar{\varphi}(N_i) \setminus \bar{\varphi}\left( V_{g}(M)\right)$ for any $a\in\mf A$.
		 		By Lemma \ref{prop:identicalcoppies:lemma}, as $p> \abs{G}$, there now exists maps $\bar{\varphi}_i:\bar{\varphi}(N_i)\rightarrow \bar{\varphi}(N_{i,1})$ such that the elements $\bar{\varphi}_i\comp\varpi_{\rho_i}(a)$ are not elements of $V_{g}(\bar{\varphi}(N_{i,1}))$. The result follows as $\bar{\varphi}(N_{i,1})$ has order $p^d\leq \left(C_{\abs G}+C_{\abs G}\ln( (\vert F \vert(n))^{\abs{\mc A}})\right)^d$ and $\vert F \vert$ is polynomial.
		 		
		 	\end{proof}
		 	
		 	The previous lemma was stated for $\ZZ$-irreducible representations. However, the same statement holds for several copies of the same irreducible representation. 
		 	\begin{lemma}\label{prop:conjMultipleCoppies:lemma}
		 		Let $\rho:G\acts M$ be a representation of a finite group on a free $\ZZ$-module such that ${M\otimes \QQ}$ has up to isomorphism a unique $\QQ$-irreducible subrepresentation $N_1$. Let $d$ be the dimension of one of the $\CC$-irreducible subrepresentations of $M$. For every $g\in G$ there exists some constant $C$ such that for any $\mf A\subset M$ with $\mf A\cap V_g(M)=\emptyset$ and $\abs{\mf A}\leq \abs{G}$, there exists a finite action-preserving quotient $\pi:M\rightarrow Q$ such that $\vert Q \vert \leq C + C \ln^d(\Vert \mf A \Vert)$ and such that $\pi(a)\notin\pi(V_g(M))$ whenever $a\in \mf A$.
		 		
		 	\end{lemma}
		 	\begin{proof}
		 	Fix any generating set $S_M$ on $M$ and an element $g \in G$. Let $M_1$ be a $\rho$-invariant finitely generated subgroup in $N_1$. Then, after rescaling $M_1$ if necessary, $M$ embeds as finite index subrepresentation of $M_1\oplus M_1\oplus\cdots\oplus M_1$, as both have the same character over $\QQ$. On $M_1\oplus\cdots\oplus M_1$ we fix a generating set $S_0$ containing at least the image of $S_M$ under this embedding. As for any $a\in\mf A$, it holds that $a\notin V_{g}(M)$, we also have that in $M_1\oplus\cdots\oplus M_1$ that ${a\notin V_g(M_1\oplus\nobreak\cdots\oplus M_1)}$. Using Lemma \ref{prop:identicalcoppies:lemma}, we obtain a finite set of $G$-invariant maps $\left\{\varphi_i:M_1\oplus\cdots\oplus M_1\rightarrow M_1 \mid i \in I \right\}$ such that $\varphi_i(\mf A)\cap V_g(M_1)=\emptyset$ for some $i \in I$.
		 	Fix a generating set $S$ containing the image $\varphi_i(S_0)$ for all $i \in I$. 	 		For this generating set on $M_1$, the map $M\rightarrow M_1\oplus\cdots\oplus M_1\rightarrow M_1$ does not increase the word norm. The result now follows immediately from Lemma \ref{prop:conjsepnotinW':lemma}.
		 	\end{proof}
		 	Note that in the previous lemmas, the size of the quotient is independent of the size of $\mf A$. Note that the bound is also sharp. Indeed this follows from the lower bound in $\cite{dere2023residual}$.

		 	\begin{proposition}
		 		Let $\rho:G\acts M$ be a representation of a finite group on a free $\ZZ$-module, and let $H<M\rtimes_\rho G$ be an extension of $\abs{G}M$ by $G$. Then $\conj_H(n)\prec \ln^{\mf k_2}(n)$ where $k$ is given by $$\mf k_2 =\max_{g\in G}\max_{(v_1,g),(v_2,g)\in H}\min\set{\dim K\mid K g,m_H\mathrm{-admits} (v_1,v_2)}$$
		 		with $m_H$ the constant from Notation \ref{not:globalconstant}.
		 	\end{proposition}
		 	\begin{proof}
		 		Let $S$ be a generating set and let $(v_1,g),(v_2,h)\in H$ be non-conjugate in $H$ and of word norm at most $n$. First if $g$ and $h$ are non-conjugate in $G$, then the quotient map $H\rightarrow G$ preserves the non-conjugacy of these elements, and $G$ is a finite quotient separating this conjugacy classes. So from now on, we assume that $g$ and $h$ are conjugate in $G$, and fix $g_0 \in G$ be such that $g_0hg_0\inv=g$. As $(0,g_0)(v_2,h)(0,g_0)\inv=(\rho(g_0)v_2,g)$, it follows thus that $$\conj \left((v_1,g),(v_2,h)\right)=\conj\left((v_1,g),(\rho(g_0)v_2,g)\right).$$ As there are only finitely many $\rho(g_0)$, which are moreover linear, we may thus restrict ourselves to non-conjugate pairs $(v_1,g)(v_2,g)$ of word norm at most $C'n$ for some $C^\prime > 0$.
We fix a generating set $S_M$ on $M$ that contains the set $\set{\rho(g)v\mid(v,h)\in S,g\in G}$, which implies that projection on the first component does not increase the norm for $S_M$. 

		 		After fixing generating sets on $M_i$, we find constants $s_i > 0$ such that the projection maps $\pi_i$ increase the word norm by at most a constant $s_i$. Similarly the maps $\rho(g)$ are linear and increase the word norm by at most a constant $s_g$.
		 		Let $C_i$ be the constant of Lemma \ref{prop:conjMultipleCoppies:lemma} where $M=M_i$.
		 	We now define the constant $$C=\max \left(\prod_{i=1}^l C_i \cdot \abs{G}^{(3+m_H)\dim{M}},2\max_{1 \leq i \leq l}(s_i)\max_{g \in G}(s_g)\right).$$
		 		
		 	 Let $K$ be a subset of $\range{l}$ of minimal dimension that $g,m_H$ admits $(v_1,v_2)$.		For $i\in K$, let $\mc H_i$ be the subset of $C(g)$, containing precisely those elements $h$ such that $(v_1,v_2)$ is $g,m_H$-globally unsolvable in $(i,h)$. Let $\varphi_i:M_i\rightarrow Q$ then be a map satisfying Lemma \ref{prop:conjMultipleCoppies:lemma} for $\mf A=\set{\pi_i(v_1-\rho(h)v_2\mid h\in \mc H_i)}$. The norm of $\pi_i(v_1-\rho(h)v_2)$ is bounded above by $2s_is_hn\leq CC'n$, and thus $Q$ is of order at most $C_i\ln^{d_i}(CC'n)$.
		 	 
		 	Let $K_i$ be the kernel of $\varphi_i\comp\pi_i$, $L_i=\pi_i\inv(\abs{G}^{m_H}M_i)$ and define $N=\left(\abs{G}^3M\right)\cap \left(\bigcap_{i=1}^l L_i \right)\cap \left(\bigcap_{i=1}^l K_i\right)$. The index of $N$ is at most the product of the indices of the components, and thus bounded by $C\ln^k{CC'n}$.
		 		Furthermore, $v_1-\rho(h)v_2-v_h+\rho(g)v_h\notin K+\abs{G}W_g(M)$, as for any $h \in C(g)$, one of the following must hold:\begin{itemize}
		 			\item $(v_1,v_2)$ is strongly unsolvable in $h$, in this case ${(v_1-\rho(h)v_2-v_h+\rho(g)v_h)\notin W_g(M)+\abs{G}^3M}$.\\
		 			\item $(v_1,v_2)$ is $g,m_H$-locally unsolvable in $(i,h)$ for some $i\in\range{l}$. In this case, $\pi_i(v_1-\rho(h)v_2)\notin\abs{G}^{m_H}M+V_g(M_i)$, and thus $v_1-\rho(h)v_2-v_h+\rho(g)v_h\notin L_i+\abs{G}W_g(M)$.\\
		 			\item $(v_1,v_2)$ is $g,m_H$-globally unsolvable in $(i,h)$ for some $i\in K$, in particular, $h\in \mc H_i$, and thus $v_1-\rho(h)v_2-v_h+\rho(g)v_h\notin K_i+\abs{G}W_g(M)$.
		 		\end{itemize}

		 		The subspace $N$ is thus such that $v_1-\rho(h)v_2 - v_h+\rho(g)v_h\notin \abs{G}W_g(M)+N$ for any choice of $h\in C(g)$. As both $N$ and $\abs{G}M$ are $\rho$-invariant, also $N'=N\cap\abs{G}M=N\cap H$ is $\rho$-invariant and thus $N'$ is a normal subgroup of $H$. This subspace is of index at most $\abs{G}C\ln^k(CC'n)$ in $H$. Furthermore, $(v_1,g)N'$ and $(v_2,g)N'$ are non-conjugate in $H / N^\prime$. Indeed if they would be conjugate, then $N$ as a subset of $M$ would have by Corollary \ref{prop:conjugatesinrepresentationquotient:corollary} some $h\in C(g)$ such that $$v_1-\rho(h)v_2 - v_h+\rho(g)v_h\in \abs{G}W_g(M)+N.$$ We conclude that $\conj_{H,S}(n)\leq \abs{G}C\ln^k(CC'n)$.

		 	\end{proof}
	\bibliographystyle{plain}
	\bibliography{refs}
\end{document}